\documentclass[12pt]{amsart}
\usepackage{geometry} 
\usepackage{amscd}
\usepackage{amssymb}
\usepackage{amsmath}
\usepackage{amsthm}

\newtheorem{thm}{Theorem}[section]
\newtheorem{lemma}[thm]{Lemma}
\newtheorem{corollary}[thm]{Corollary}
\newtheorem{prop}[thm]{Proposition}

\theoremstyle{definition}

\newtheorem{rem}[thm]{Remark}

\newtheorem{defn}[thm]{Definition}

\makeatletter

\newcommand{\isom}{\overset{\sim}{\rightarrow}}

\title{Zariski density of crystalline representations for any $p$-adic field.}
\author{Kentaro Nakamura}
\date{} 

\begin{document}

\maketitle
\pagestyle{plain}
\footnote{2010 Mathematical Subject Classification 11F80 (primary), 11F85 (secondary).
}
\begin{abstract}
The aim of this article is to prove Zariski density of crystalline representations 
in the rigid analytic space associated to the universal deformation ring 
of a $d$-dimensional mod $p$ representation of $\mathrm{Gal}(\overline{K}/K)$ for any $d$ and 
any $p$-adic field $K$. 
This is a generalization of the results of Colmez, Kisin for $d=2$ and $K=\mathbb{Q}_p$, 
of the author for $d=2$ and any $K$, and of Chenevier for any $d$ and $K=\mathbb{Q}_p$. 
A key ingredient for the proof is to construct a $p$-adic family of trianguline representations which can be seen as a local analogue of eigenvarieties. 
In this article, we construct such a family by generalizing Kisin's theory of finite slope subspace $X_{fs}$ for 
any $d$ and any $K$, and using Bella\"iche-Chenevier's idea of using exterior products in the study of trianguline deformations.

\end{abstract}
\setcounter{tocdepth}{2}
\tableofcontents

\section{Introduction.}
\subsection{Background.}
Let $K$ be a finite extension of $\mathbb{Q}_p$ and $d\in \mathbb{Z}_{\geqq 1}$. Let $E$ be a sufficiently large finite extension of $\mathbb{Q}_p$ with the ring of integer
$\mathcal{O}$ and the residue field $\mathbb{F}$. Let 
$\overline{V}$ be a $\mathbb{F}$-representation of $G_K:=\mathrm{Gal}(\overline{K}/K)$ of rank $d$,  i.e. a $d$-dimensional 
$\mathbb{F}$-vector space with a continuous $\mathbb{F}$-linear $G_K$-action. Let $\mathcal{C}_{\mathcal{O}}$ be the category of 
Artin local $\mathcal{O}$-algebras with residue field $\mathbb{F}$. We consider the deformation functor $D_{\overline{V}}:\mathcal{C}_{\mathcal{O}}\rightarrow (Sets)$ defined by $D_{\overline{V}}(A):=\{$equivalent classes of deformations 
of $\overline{V}$ over $A\}$ for $A\in\mathcal{C}_{\mathcal{O}}$. Assume that $\mathrm{End}_{\mathbb{F}[G_K]}(\overline{V})=\mathbb{F}$, 
then $D_{\overline{V}}$ is 
representable by the universal deformation ring $R_{\overline{V}}$.
 Let $\mathcal{X}_{\overline{V}}$ be the rigid analytic space associated 
to $R_{\overline{V}}$.  The points of $\mathcal{X}_{\overline{V}}$ correspond to 
$p$-adic representations of $G_K$ with mod $p$ reduction isomorphic to $\overline{V}$. Define the subset $\mathcal{X}_{\overline{V},\mathrm{reg}-\mathrm{cris}}$ of $\mathcal{X}_{\overline{V}}$ by 
\begin{multline*}
\mathcal{X}_{\overline{V},\mathrm{reg-\mathrm{cris}}}:=\{x=[V_x]\in \mathcal{X}_{\overline{V}}|V_x \text{ is crystalline with Hodge-Tate weights }\\
\{k_{i,\sigma}\}_{1\leqq i\leqq d,\sigma:\mathbb{Q}_p\hookrightarrow\overline{\mathbb{Q}}_p} 
 \text{ such that  }  k_{i,\sigma}\not=k_{j,\sigma} 
 \text{ for any } i\not=j \text{ and }\sigma:K\hookrightarrow \overline{\mathbb{Q}}_p\}
   \end{multline*}
We denote by $\overline{\mathcal{X}}_{\overline{V},\mathrm{reg}-\mathrm{cris}}$ the Zariski closure of $\mathcal{X}_{\overline{V},\mathrm{reg}-\mathrm{cris}}$ in $\mathcal{X}_{\overline{V}}$. The main results of this article which will be proved in $\S4$ concern with Zariski density of $\mathcal{X}_{\overline{V},\mathrm{reg}-\mathrm{cris}}$ in $\mathcal{X}_{\overline{V}}$.
For example, one of the main results of this article is following (see Corollary \ref{1.18}).

\begin{thm}\label{0.1}
If $\overline{V}$ is absolutely irreducible and satisfies $(i)$ $p\not|d$ and $\zeta_p\in K$, or $(ii)$ $\overline{V}\not\isom \overline{V}(\omega)$, then 
we have an equality 
$$\overline{\mathcal{X}}_{\overline{V},\mathrm{reg}-\mathrm{cris}}=\mathcal{X}_{\overline{V}},$$
where $\omega:G_K^{\mathrm{ab}}\rightarrow \mathbb{F}^{\times}$ is the mod $p$ cyclotomic character. 

\end{thm}

This theorem is a generalization 
of the results of Colmez, Kisin \cite{Co08}, \cite{Ki10} for $d=2$ and $K=\mathbb{Q}_p$,
of the author \cite{Na10} for $d=2$ and any $K$, and of Chenevier  \cite{Ch13} for any $d$ and
$K=\mathbb{Q}_p$. 
When $d=2$ and $K=\mathbb{Q}_p$, the result of Colmez and Kisin plays some crucial roles in their studies of 
$p$-adic Langlands correspondence for $\mathrm{GL}_2(\mathbb{Q}_p)$.

The idea of the proof is essentially the same as those of \cite{Co08}, \cite{Ki10}, \cite{Na10}, \cite{Ch13},
i.e. we re-interpret purely locally the argument of infinite fern of Gouv\^ea-Mazur by using the notion of trianguline representations. 
Inspired by the work  of Kisin \cite{Ki03} on a $p$-adic Hodge theoretic study of Coleman-Mazur eigencurve, where 
he essentially proved that the restrictions to $G_{\mathbb{Q}_p}$ of the two dimensional $p$-adic representations of $G_{\mathbb{Q}}$ parametrized by Coleman-Mazur eigencurve are trianguline, Colmez \cite{Co08} defined and studied trianguline representations for $K=\mathbb{Q}_p$ by using the theory of $(\varphi,\Gamma)$-modules over the Robba ring.  In \cite{Na09}, the author of 
this article generalized Colmez's results, i.e. studied trianguline representations for general $K$ by using the theory of $B$-pair 
defined by Berger \cite{Be08}.

There are two key ingredients for the proof of the main theorem. One is the deformation theory of trianguline representations, and the other is the construction of a ``universal" $p$-adic family of trianguline representations in which the set of the crystalline points is 
Zariski dense.

For the deformation theory of trianguline representations, we have already obtained satisfying results in \cite{BeCh09}, \cite{Ch11} for 
$K=\mathbb{Q}_p$ and in \cite{Na10} for general $K$.

The other one (i.e. the construction of a $p$-adic family of trianguline representations) is more important, which can be  seen as a construction of local analogue of eigenvarieties.
For $K=\mathbb{\mathbb{Q}}_p$ and $d=2$, two different constructions by Colmez and Kisin are known. 
Colmez \cite{Co08} explicitly constructed (more generally) a $p$-adic family of rank two trianguline $(\varphi,\Gamma)$-modules over the Robba ring by explicitly calculating the cohomology of some rank one $(\varphi,\Gamma)$-modules over the relative Robba ring of affinoid algebras.
On the other hands, Kisin \cite{Ki03} constructed a Zariski closed subspace $X_{fs}$ of $\mathcal{X}_{\overline{V}}\times_E\mathbb{G}_{m/E}^{\mathrm{an}}$ using his theory of the finite slope subspace, which is 
(roughly) defined as the subspace consisting of the points $([V],\lambda)\in 
\mathcal{X}_{\overline{V}}\times_E\mathbb{G}_{m/E}^{\mathrm{an}}$ such that $\bold{D}^+_{\mathrm{crys}}(V)^{\varphi=\lambda}\not=0$, and showed that the family of 
$p$-adic Galois representation on this subspace is a universal (in some sense) $p$-adic family of trianguline representations. For any $d\in \mathbb{Z}_{\geqq 1}$ but for $K=\mathbb{Q}_p$, Chenevier \cite{Ch13} 
recently generalized Colmez's construction and constructed 
a universal $p$-adic family of rank $d$ trianguline $(\varphi,\Gamma)$-modules by further developing 
the cohomology theory of $(\varphi,\Gamma)$-modules over the relative Robba ring of affinoid algebras. Because his calculation of cohomologies heavily depends on the explicit structure of $(\varphi,\Gamma)$-modules which is available only for  $K=\mathbb{Q}_p$, we cannot directly generalize his results for general $K$. The main feature of this article is to construct 
$p$-adic families of trianguline representations for any $d$ and for any $K$ by 
generalizing Kisin's theory of finite slope subspace. For $d=2$, we have already done this in \cite{Na10}. To generalize the construction of \cite{Na10} for higher dimensional case, we use an idea of Bella\"iche-Chenevier of using exterior products in the study of trianguline deformations. We explain our results in detail below.

\subsection{Overview}
Here, we give an overview of the contents of this article. 

In$ \S$ 2, we first recall the fundamentals of trianguline representations using the notion of $B$-pairs. The notion of $B$-pairs was defined by Berger \cite{Be08}. 
The category of 
$E$-representation of $G_K$ can be naturally embedded in the category of $E$-$B$-pairs 
of $G_K$. For an $E$-representation $V$, we denote by $W(V)$ the associated $E$-$B$-pair.
We say that $V$ is a split trianguline $E$-representation if $W(V)$ can be written as a successive extension 
of rank one $E$-$B$-pairs, i.e. there exists a filtration $T:0\subseteq W_1\subseteq W_2\subseteq 
\cdots\subseteq W_d=W(V)$ by $E$-$B$-pairs $W_i$ such that $W_i/W_{i-1}$ are rank one $E$-$B$-pairs 
for any $i$.  We call $T$ a triangulation of $V$. Rank one $E$-$B$-pairs can be classified  by 
the set of continuous homomorphisms $\delta:K^{\times}\rightarrow E^{\times}$ (\cite{Co08}, \cite{Na09}). 
For a continuous homomorphism $\delta:K^{\times}\rightarrow E^{\times}$, we denote by $W(\delta)$ the rank one 
$E$-$B$-pair associated to  $\delta$. By the definition of $T$, there exists a set $\{\delta_i\}_{i=1}^d$ where 
$\delta_i:K^{\times}\rightarrow E^{\times}$ 
such that $W_i/W_{i-1}\isom W(\delta_i)$ for each $i$, which we call the parameter of $T$.
Therefore, to construct a $p$-adic 
family of trianguline representations, we first need to construct a universal $p$-adic family of continuous homomorphisms 
$\delta:K^{\times}\rightarrow E^{\times}$. 
Let $\mathcal{T}$ and $\mathcal{W}$ be the rigid analytic spaces over $E$ which represent the functors defined by
$\mathcal{T}(A):=\{\delta:K^{\times}\rightarrow A^{\times}$ continuous homomorphism $\}$ and 
$\mathcal{W}(A):=\{\delta:\mathcal{O}_{K}^{\times}\rightarrow A^{\times}$continuous homomorphism $\}$ for each $E$-affinoid $A$. 
If we fix a uniformizer $\pi_K$ of $K$, we have an isomorphism $\mathcal{T}\isom \mathcal{W}\times_E \mathbb{G}_{m/E}^{\mathrm{an}}
:\delta\mapsto ( \delta|_{\mathcal{O}_K^{\times}},\delta(\pi_K))$. For any $\delta\in \mathcal{W}(A)$ (which is known to be automatically $\mathbb{Q}_p$-analytic), we denote $k(\delta)_{\sigma}:=\frac{\partial\delta(x)}{\partial\sigma(x)}|_{x=1}\in A$ for any embedding $\sigma:K\hookrightarrow E$, which we call the $\sigma$-part of Hodge-Tate weight of $\delta$. For $\delta\in \mathcal{W}(A)$, 
define 
$\widetilde{\delta}:G_K^{\mathrm{ab}}\rightarrow A^{\times}$ the character such that 
$\widetilde{\delta}\circ\mathrm{rec}_K|_{\mathcal{O}_K^{\times}}=\delta$ and $\widetilde{\delta}(\mathrm{rec}_K(\pi_K))=1$, where $\mathrm{rec}_K:K^{\times}\hookrightarrow G_K^{\mathrm{ab}}$ is the reciprocity map of local class field theory. For $\delta\in \mathcal{T}(A)$, we also define $\widetilde{\delta}:=\widetilde{(\delta|_{\mathcal{O}_K^{\times}})}$.

In \cite{Na10}, we modified and generalized Kisin's finite slope subspace $X_{fs}\subseteq\mathcal{X}_{\overline{V}}\times_E\mathbb{G}_{m/E}^{\mathrm{an}}$ for any $K$, i.e. twisting by a universal character on $\mathcal{W}$, we constructed  $X_{fs}$ as a Zariski closed subspace of $\mathcal{X}_{\overline{V}}\times_E\mathcal{T}$ instead of $\mathcal{X}_{\overline{V}}\times_E \mathbb{G}_{m/E}^{\mathrm{an}}$. In this article, we generalize the construction of \cite{Na10} for any $d$, we construct $X_{fs}$ which we denote by $\mathcal{E}_{\overline{V}}$, as a Zariski closed subspace of 
$\mathcal{Z}:=\mathcal{X}_{\overline{V}}\times_E\mathcal{T}^{\times (d-1)}$. Let $V$ be a split trianguline $E'$-representation with 
a triangulation $T$ whose parameter is $\{\delta_i\}_{i=1}^d$ such that $[V]\in \mathcal{X}_{\overline{V}}(E')$ for a finite extension $E'$ of $E$. From such a pair $(V,T)$, we define an $E'$-rational 
point $z_{(V,T)}:=([V],\delta_1,\delta_2,\cdots,\delta_{d-1})\in\mathcal{Z}(E')$. For the point 
$z_{(V,T)}$, we define $\delta_d:=(\mathrm{det}(V)\circ \mathrm{rec}_K)\cdot\prod_{i=1}^{d-1}\delta_i^{-1}$. For any $n\in\mathbb{Z}_{\geqq 1}$, set 
$[n]:=\{1,2,\cdots,n\}$. For any 
subset $I\subseteq [d]$, set $\delta_{I}:=\prod_{i\in I}\delta_i$.

The space $\mathcal{E}_{\overline{V}}$ should be a suitable approximation of the subset 
of $\mathcal{Z}$ consisting of all the points of the form $z_{(V, T)}$. Hence, the space $\mathcal{E}_{\overline{V}}$ should be contained in the closed subspace $\mathcal{Z}_0$ 
of $\mathcal{Z}$ consisting of the points $([V],\delta_1,\cdots,\delta_{d-1})$ such that the Hodge-Tate 
weights of $V$ are compatible with those of $\{\delta_i\}_{i=1}^d$, more precisely, the 
$\sigma$-part of Hodge-Tate wights of $V$ is equal to $\{k(\delta_1)_{\sigma},\cdots, 
k(\delta_{d-1})_{\sigma}, k(\delta_d)_{\sigma}\}$ for any $\sigma:K\hookrightarrow E$. The basic idea of the construction of $\mathcal{E}_{\overline{V}}$ as the closed subspace of $\mathcal{Z}_0$ is to generalize Kisin's construction of $X_{fs}$ by using a technique of Bella\"iche-Chenevier using exterior products in the study of 
trianguline deformations (\cite{BeCh09}, see proposition 
\ref{1.8}). If $V$ is a split trianguline $E$ representation as above. 
Then, we can show that (*)``$\bold{D}^+_{\mathrm{cris}}((\wedge^iV)(\widetilde{\delta}_{[i]}^{-1}))^{\varphi^f=\delta_{[i]}(\pi_K)}$ is
non-zero for any $1\leqq i\leqq d-1$". We construct $\mathcal{E}_{\overline{V}}$ as a subspace of $\mathcal{Z}_0$ roughly parametrizing the points 
$z=(V,\delta_1,\cdots,\delta_{d-1})$ with this property (*). 

The key main theorem of this article is the following. For the precise statements and definitions, see Corollary \ref{1.3}, Proposition \ref{1.4} and Theorem \ref{1.7}. 
\begin{thm}\label{0.2}
There exists a Zariski closed subspace $\mathcal{E}_{\overline{V}}$ of $\mathcal{Z}$ satisfying the following properties (0), (1), (2).
\begin{itemize}
\item[(0)]For any point $z=([V],\delta_1,\cdots,\delta_{d-1})\in \mathcal{E}_{\overline{V}}(E')$, the $\sigma$-part of Sen's polynomial 
of $V$ is equal to $$\prod_{i=1}^d\bigl(T-k(\delta_i)_{\sigma}\bigr) \in E'[T]$$ 
for each embedding $\sigma:K\hookrightarrow E$.

\item[(1)]
If a pair $(V,T)$ as above satisfies the following conditions $\mathrm{(i)}$, $\mathrm{(ii)}$, $\mathrm{(iii)}$,

\begin{itemize}
\item[(i)]$\mathrm{End}_{E'[G_K]}(V)=E'$,
\item[(ii)]$\delta_{j}/\delta_i\not=\prod_{\sigma\in \mathcal{P}}\sigma^{k_{\sigma}}$ for any $i<j$ and 
$\{k_{\sigma}\}_{\sigma\in \mathcal{P}}\in \prod_{\sigma\in \mathcal{P}}\mathbb{Z}_{\leqq 0}$, 
\item[(iii)]$\delta_i/\delta_j\not=|N_{K/\mathbb{Q}_p}|_p\prod_{\sigma\in \mathcal{P}}\sigma^{k_{\sigma}}$ for $i<j$ and 
$\{k_{\sigma}\}_{\sigma\in \mathcal{P}}\in \prod_{\sigma\in \mathcal{P}}\mathbb{Z}_{\geqq 1}$,
\end{itemize}
then the point $z_{(V,T)}\in\mathcal{Z}$ defined above is contained in $\mathcal{E}_{\overline{V}}$. 
\item[(2)]If the point $z_{(V,T)}\in \mathcal{E}_{\overline{V}}$ defined in $(1)$ satisfies one of the following additional conditions  $\mathrm{(iv)}$,  $\mathrm{(v)}$, 
\begin{itemize}
\item[(iv)]$V$ is potentially crystalline and, for any $1\leqq i\leqq d-1$, $$\{a\in \bold{D}_{\mathrm{cris}}((\wedge^i V)(\widetilde{\delta}_{[i]}^{-1}))| 
\exists n\geqq 1\text{ such that }  (\varphi^f-\delta_{[i]}(\pi_K))^na=0\}$$ is a rank one 
free $K_0\otimes_{\mathbb{Q}_p}E$-module,
\item[(v)]for any $1\leqq i\leqq d-1$ and for any subset $I(\not=[i])\subseteq [d]$ with its 
cardinality equal to $i$, 
we have  
$$k(\delta_{I})_{\sigma} -k(\delta_{[i]})_{\sigma}\not\in \mathbb{Z}_{\leqq 0}$$ for any $\sigma\in \mathcal{P}$, 

\end{itemize}
then 
$\mathcal{E}_{\overline{V}}$ is smooth at $z_{(V,T)}$ of its dimension $[K:\mathbb{Q}_p]\frac{d(d+1)}{2}+1$.
\end{itemize}

\end{thm}

For eigenvarieties, the classicality theorem concerning the conditions for 
overconvergent modular forms to be classical is very important, in particular, which enables us 
to show that 
the set of classical points is Zariski dense in eigenvarieties. As a 
local analogue of this property, we prove the following theorem (Theorem \ref{1.13}).

\begin{thm}\label{0.3}
Let $(V,T)$ be a pair satisfying all the conditions in (1) and the condition (iv) or (v) in (2) of Theorem \ref{0.2}, and 
let $U$ be an admissible open neighborhood of $z:=z_{(V,T)}$ in $\mathcal{E}_{\overline{V}}$. 
Then there exists an admissible open neighborhood $U'$
of $z$ in $U$ in which the subset consisting of the points $z'=([V'],\delta'_1,\cdots,\delta'_{d-1})\in U'$ such that $V'$ is crystalline with distinct Hodge-Tate weights are Zariski dense in $U'$. 

\end{thm}
For the proof of Theorem \ref{0.3}, we need to prove that, oppositely, if a point 
$z=([V],\delta_1,\cdots,\delta_{d-1})\in \mathcal{E}_{\overline{V}}$ satisfying the condition (*) above and some conditions on 
$\{\delta_i\}_{i=1}^d$, then $V$ is split trianguline and crystalline. Concerning this problem, we prove some propositions (see Proposition \ref{1.11}, Proposition \ref{1.12}) using the slope filtration theorem of Kedlaya.

The main theorem Theorem \ref{0.1} follows from these two theorems Theorem \ref{0.2}, Theorem \ref{0.3}, and from the deformation theory of trianguline representations, 
in particular, the deformation theory of generic or benign crystalline representations developed 
in \cite{Ch11} and \cite{Na10}. 

\subsection*{Acknowledgement.}
A part of this article was written during a stay at \'Ecole Polytechnique. 
The author would like to thank Ga\"etan Chenevier for many valuable discussions and for his kind hospitality, and thanks also all the people there for their kind hospitality. The author would like to thank Kenichi Bannai for constantly encouraging the author, and would like to thank Go Yamashita and Seidai Yasuda for many valuable discussions and pointing out some mistakes in the previous version of this article.
The author would like to thank all the people at Keio University related to the JSPS ITP program which enabled my stay  at 
\'Ecole Polytechnique. This work is also supported in part by the Grant-in-aid (NO. S-23224001) for Scientific Research, JSPS.

\subsection*{Notation.}
Let $p$ be a prime number. Let $K$ be a finite extension of $\mathbb{Q}_p$, $\overline{K}$ 
the algebraic closure of $K$, $K_0$ the maximal unramified extension of $\mathbb{Q}_p$ in $K$.
Let $G_K:=\mathrm{Gal}(\overline{K}/K)$ be the 
absolute Galois group of $K$ equipped with pro-finite topology. Let $\mathcal{O}_K$ be the ring of integers of $K$, $\pi_K\in\mathcal{O}_K$ 
a fixed uniformizer of $K$, $k:=\mathcal{O}_K/\pi_K\mathcal{O}_K$ the residue field 
of $K$ with cardinality $q=p^f$. Let $\chi_p:G_K\rightarrow \mathbb{Z}_p^{\times}$ be the $p$-adic cyclotomic character 
(i.e. $g(\zeta_{p^n})=\zeta_{p^n}^{\chi(g)}$ for any $p^n$-th roots of unity and for any $g\in G_K$). 
Let $E$ be a finite extension of $\mathbb{Q}_p$ such 
that $\mathcal{P}:=\mathrm{Hom}_{\mathbb{Q}_p-alg}(K, \overline{E})=\mathrm{Hom}_{\mathbb{Q}_p-alg}(K,E)$. 
In this paper, we use the notation $E$ for the coefficient field of representations. 
Let $|-|_p:E\rightarrow \mathbb{Q}_{\geqq 0}$ be the norm such that $|p|_p:=\frac{1}{p}$. 
Let $N_{K/\mathbb{Q}_p}:K^{\times}\rightarrow \mathbb{Q}_p^{\times}$ be the norm. 
Let denote the composition by $|N_{K/\mathbb{Q}_p}|_p:K^{\times}\xrightarrow{N_{K/\mathbb{Q}_p}}\mathbb{Q}_p^{\times}\xrightarrow{|-|_p}\mathbb{Q}^{\times}\hookrightarrow E^{\times}$, where the last inclusion is the canonical one.
Let $\chi_{\mathrm{LT}}:G_K\rightarrow \mathcal{O}_K^{\times}$ be 
the Lubin-Tate character associated to the fixed uniformizer $\pi_K$. Let $\mathrm{rec}_K:K^{\times}
\rightarrow G_K^{\mathrm{ab}}$ be the reciprocity map of local class field theory such that 
$\mathrm{rec}_K(\pi_K)$ is a lifting of the inverse of $q$-th power Frobenius on $k$, then we have 
$\chi_{\mathrm{LT}}(\mathrm{rec}_K(\pi_K))=1$ and $\chi_{\mathrm{LT}}\circ\mathrm{rec}_K|_{\mathcal{O}_K^{\times}}
=\mathrm{id}_{\mathcal{O}_K^{\times}}$. For any topological ring $A$,
we say that $V_A$ is an $A$-representation of $G_K$ 
if $V_A$ is a finite free $A$-module with a continuous $A$-linear $G_K$-action.

\section{Review of $B$-pairs and trianguline representations}

In this section, we recall the definition of $B$-pairs and trianguline representations and 
some of their fundamental properties which we will use in later sections, see \cite{Be08} or \cite{Na09}, \cite{Na10} for more details.

\subsection{review of trianguline $A$-$B$-pairs}

Let $\bold{B}_{\mathrm{cris}}$, $\bold{B}^+_{\mathrm{dR}}$ and $\bold{B}_{\mathrm{dR}}$ be the Fontaine's rings of $p$-adic periods (\cite{Fo94}). 
Let $\bold{B}_e:=\bold{B}_{\mathrm{cris}}^{\varphi=1}$ the $\varphi$-fixed part of $\bold{B}_{\mathrm{cris}}$. These rings are naturally equipped with 
continuous $G_K$-actions. Let $t=\mathrm{log}[\varepsilon]\in \bold{B}^{\varphi=p}_{\mathrm{cris}}\cap \mathrm{Fil}^1\bold{B}^+_{\mathrm{dR}}$ be a period of the inverse of the $p$-adic cyclotomic character $\chi_p$.

Let $\mathcal{C}_E$ be the category of Artin local $E$-algebras $A$ such that 
$E\isom A/\mathfrak{m}_A$ where $\mathfrak{m}_A$ is the maximal ideal of $A$. 
For $A\in \mathcal{C}_E$, we recall the definition of $A$-$B$-pair which is the $A$-coefficient version of $B$-pair (see Definition 2.10 and Lemma 2.11 of \cite{Na10}). 

\begin{defn}
We say that a pair $W:=(W_e,W^+_{\mathrm{dR}})$ is an $A$-$B$-pair (of $G_K$) if 
\begin{itemize}
\item[(1)] $W_e$ is a finite free $\bold{B}_e\otimes_{\mathbb{Q}_p}A$-module with a continuous semi-linear $G_K$-action, where ``semi-linear" means that we have $g((a\otimes b)x)=
(g(a)\otimes b)g(x)$ for any $g\in G_K$, $a\in \bold{B}_e$, $b\in A$, $x\in W_e$.
\item[(2)]$W^+_{\mathrm{dR}}$ is a $G_K$-stable finite free sub $\bold{B}^+_{\mathrm{dR}}\otimes_{\mathbb{Q}_p}A$-module of 
$W_{\mathrm{dR}}:=\bold{B}_{\mathrm{dR}}\otimes_{\bold{B}_e}W_e$ which generates $W_{\mathrm{dR}}$ as a $\bold{B}_{\mathrm{dR}}$-module.
\end{itemize}
We define the rank of $W$ as the rank of $W_e$ over $\bold{B}_e\otimes_{\mathbb{Q}_p}A$.
We just call an $A$-$B$-pair if there is no risk of confusing about $K$.
\end{defn}
\begin{rem}
The functor $V_A\mapsto W(V_A):=(\bold{B}_e\otimes_{\mathbb{Q}_p}V_A, \bold{B}^+_{\mathrm{dR}}\otimes_{\mathbb{Q}_p}V_A)$ 
from the category of $A$-representations of $G_K$ to the category of $A$-$B$-pairs is exact and fully faithful.

\end{rem}

\begin{prop}\label{4.1}
There exists a canonical bijection $\delta\mapsto W(\delta)$ between the set of continuous 
homomorphisms $\delta:K^{\times}\rightarrow A^{\times}$ and the set of isomorphism classes of rank one $A$-$B$-pairs.
\end{prop}
\begin{proof}
See Proposition 2.16 of \cite{Na10}

\end{proof}

By the definition of $W(\delta)$ in the previous paragraph of Proposition 2.16 of \cite{Na10}, 
it is easy to show that $W(\delta)$ satisfies the following two properties.

\begin{rem}
The bijection in Proposition \ref{4.1} is compatible with local class field theory, i.e. we have an isomorphism $W(\widetilde{\delta}\circ\mathrm{rec}_K)\isom W(A(\widetilde{\delta}))$
for any character 
$\widetilde{\delta}:G_K^{\mathrm{ab}}\rightarrow A^{\times}$.
\end{rem}

\begin{rem}
For $\lambda\in A^{\times}$, we define a continuous homomorphism 
$\delta_{\lambda}:K^{\times}\rightarrow A^{\times}$ such that $\delta_{\lambda}|_{\mathcal{O}_K^{\times}}=1$ and 
$\delta_{\lambda}(\pi_K)=\lambda$. Then, $W(\delta_{\lambda})$ is a crystalline $A$-$B$-pair corresponding to 
an $A$-filtered $\varphi$-module $D_{\lambda}:=K_0\otimes_{\mathbb{Q}_p}Ae_{\lambda}$ such that 
$\varphi^f(e_{\lambda})=\lambda e_{\lambda}$ and $\mathrm{Fil}^0(K\otimes_{K_0}D_{\lambda})=K\otimes_{K_0}D_{\lambda}$ and 
$\mathrm{Fil}^1(K\otimes_{K_0}D_{\lambda})=0$.

\end{rem}

\begin{defn}
Let $W$ be an $A$-$B$-pair of rank $d$. We say that $W$ is a split trianguline $A$-$B$-pair if there exists a
 filtration $T:0\subseteq W_1\subseteq W_2\subseteq \cdots\subseteq W_{d-1}\subseteq W_d=W$ by $A$-$B$-pairs
 such that $W_i/W_{i-1}$ is rank one $A$-$B$-pair for any $i$.
 We call $T$ an $A$-triangulation of $W$. Define the set $\{\delta_i\}_{i=1}^d$ of 
 continuous homomorphisms $\delta_i:K^{\times}\rightarrow A^{\times}$ such that
  $W_i/W_{i-1}\isom W(\delta_i)$ for any $i$, which we call the parameter of $T$.
 
 Let $V$ be an $A$-representation. We say that $V$ is a split trianguline $A$-representation if 
 $W(V)$ is a split trianguline $A$-$B$-pair.

\end{defn}

\subsection{review of deformation theory of trianguline representations}

In \cite{BeCh09}, \cite{Ch11} (for $K=\mathbb{Q}_p$) and \cite{Na10} (for general $K$), we study 
deformation theory of trianguline $B$-pairs or trianguline $(\varphi,\Gamma)$-modules over the Robba ring, which we now briefly recall (see \S 2 of \cite{Na10} for more details).

Let $V$ be an $E$-representation of rank $d$ and $A\in \mathcal{C}_E$. 
We say that the pair $(V_A,\psi_A)$ is a deformation of $V$ over $A$ if 
$V_A$ is an $A$-representation and $\psi_A:V_A\otimes_A E\isom V$ is an isomorphism 
of $E$-representations. Let $(V_A,\psi_A)$ and $(V'_A,\psi'_A)$ be two deformations of $V$ over $A$, we say that 
$(V_A,\psi_A)$ and $(V'_A,\psi'_A)$ are  equivalent if there exists an isomorphism $f:V_A\isom V'_A$ of 
$A$-representations such that $\psi_A=\psi'_A\circ (f\otimes_A \mathrm{id}_{E})$. 
We define a functor $D_V:\mathcal{C}_E\rightarrow (Sets)$ by 
$$D_V(A):=\{\text{ equivalent classes of deformations of } V \text{ over } A\}$$
for $A\in \mathcal{C}_E$.

Next, we consider the pair  $(V,T)$ where $V$ is a split trianguline $E$-representation with a triangulation 
$T:0\subseteq W_1\subseteq W_2\subseteq \cdots \subseteq W_d=W(V)$. For $A\in \mathcal{C}_E$, we say that the triple $(V_A,\psi_A,T_A)$ is a trianguline 
deformation of $(V,T)$ over $A$ if $(V_A,\psi_A)$ is a deformation of $V$ over $A$ and 
$$T_A:0\subseteq W_{1,A}\subseteq W_{2,A}\subseteq \cdots \subseteq W_{d,A}=W(V_A)$$ is 
an $A$-triangulation of $V_A$ such that $$W(\psi_A)(W_{i,A}\otimes_A E)=W_i$$ for any 
$1\leqq i\leqq d$, where $W(\psi_A): W(V_A)\otimes_A E\isom W(V)$ is the isomorphism induced from $\psi_A$.
We say that two trianguline deformations $(V_A,\psi_A,T_A)$ and 
$(V'_A,\psi'_A,T'_A)$  over $A$ are equivalent if there exists an isomorphism $f:V_A\isom V'_A$ of $A$-representations 
such that $\psi_A=\psi'_A\circ (f\otimes_A \mathrm{id}_{E})$ and $W(f)(W_{i,A})=W'_{i,A}$ for any 
$1\leqq i\leqq d$. We define a functor $D_{V,T}\rightarrow (Sets)$ by 
$$D_{V,T}(A):=\{ \text{ equivalent classes of trianguline deformations of } (V,T) \text{ over } A\}$$
for $A\in \mathcal{C}_E$. Later, we simply write $[V_A]\in D_V(A)$ or $[(V_A,T_A)]\in D_{V,T}(A)$ instead of $[(V_A,\psi_A)]$ or $[(V_A,\psi_A,T_A)]$ if there is no risk of confusing about $\psi_A$.

We have a morphism of functors $D_{V,T}\rightarrow D_V$ defined by 
$[(V_A,T_A)]\mapsto [V_A]$. If $D_V$ and $D_{V,T}$ are 
represented by $R_V$ and $R_{V,T}$, then this morphism is given by a map $R_V\rightarrow R_{V,T}$, 
which is a surjection in many cases. For the representability and other properties of $D_{V,T}$, we have the following proposition.

\begin{prop}\label{4.2}
Let $V$ be a split trianguline $E$-representation with a triangulation $\mathcal{T}$ whose parameter 
is $\{\delta_i\}_{i=1}^d$. We assume that $(V,T)$ satisfies the following conditions, 
\begin{itemize}
\item[(i)]$\mathrm{End}_{E[G_K]}(V)=E$ (then $D_V$ is representable),
\item[(ii)]$\delta_{j}/\delta_i\not=\prod_{\sigma\in \mathcal{P}}\sigma^{k_{\sigma}}$ for any $i<j$ and 
$\{k_{\sigma}\}_{\sigma\in \mathcal{P}}\in \prod_{\sigma\in \mathcal{P}}\mathbb{Z}_{\leqq 0}$, 

\end{itemize}
then the functor $D_{V,T}$ is representable by a quotient $R_{V,T}$ of $R_V$.
Moreover, if  $(V,T)$ satisfies the following additional condition,
\begin{itemize}
\item[(iii)]$\delta_i/\delta_j\not=|N_{K/\mathbb{Q}_p}|_p\prod_{\sigma\in \mathcal{P}}\sigma^{k_{\sigma}}$ for any $i<j$ and $\{k_{\sigma}\}_{\sigma\in \mathcal{P}}\in \prod_{\sigma\in \mathcal{P}}\mathbb{Z}_{\geqq 1}$,
\end{itemize}
then $R_{V,T}$ is formally smooth over $E$ of its dimension 
$[K:\mathbb{Q}_p]\frac{d(d+1)}{2}+1$.

\end{prop}
\begin{proof}
See \cite{BeCh09} and Corollary 2.30, Lemma 2.48 and Proposition 2.39 of \cite{Na10}.
\end{proof}

Next, we recall some relations between crystalline representations and trianguline representations.
Let $V$ be a crystalline $E$-representation of rank $d$. We define the crystalline deformation functor 
$D_V^{\mathrm{cris}}$ which is a subfunctor of $D_V$ defined by 
$$D_{V}^{\mathrm{cris}}(A):=\{[V_A]\in D_{V}(A)| V_A \text{ is crystalline } \}$$ 
for $A\in \mathcal{C}_E$. The natural inclusion $D^{\mathrm{cris}}_V\hookrightarrow D_V$ is 
relatively representable, and $D_V^{\mathrm{cris}}$ is formally smooth (\cite{Ki08}).

Let $\bold{D}_{\mathrm{cris}}(V):=(\bold{B}_{\mathrm{cris}}\otimes_{\mathbb{Q}_p}V)^{G_K}$ be 
the filtered $\varphi$-module associated to $V$, which is a finite free $K_0\otimes_{\mathbb{Q}_p}E$-module of rank $d$. 
Let $\{\alpha_1,\alpha_2,\cdots,\alpha_d\} \subseteq \overline{E}$ be the set (possibly with multiplicity) of eigenvalues of $\varphi^f$ ($f:=[K_0:\mathbb{Q}_p]$) 
 on $\bold{D}_{\mathrm{cris}}(V)\otimes_{K_0\otimes_{\mathbb{Q}_p}E,\sigma\otimes\mathrm{id}_E}\overline{E}$ for one $\sigma:K_0\hookrightarrow 
\overline{E}$, which does not depend on the choice of $\sigma$. We assume that $\alpha_i\not=\alpha_j$ for any $i\not=j$. Extending scalars, we assume that $\{\alpha_1,\cdots,\alpha_d\}\subseteq E$ and $\bold{D}_{\mathrm{cris}}(V)$ can be written as
$$\bold{D}_{\mathrm{cris}}(V)=K_0\otimes_{\mathbb{Q}_p}Ee_1\oplus K_0\otimes_{\mathbb{Q}_p}Ee_2\oplus \cdots\oplus 
K_0\otimes_{\mathbb{Q}_p}Ee_d$$ such that $K_0\otimes_{\mathbb{Q}_p}Ee_i$ is $\varphi$-stable and $\varphi^f(e_i)=\alpha_ie_i$ 
for any $1\leqq i\leqq d$. Let $\mathfrak{S}_d$ be the $d$-th permutation group.
Under these assumptions, we define a filtration as filtered $\varphi$-modules
$$\mathcal{F}_{\tau}:0\subseteq D_{\tau,1}\subseteq D_{\tau,2}\subseteq 
\cdots \subseteq D_{\tau,d}=\bold{D}_{\mathrm{cris}}(V)$$ for each 
$\tau\in \mathfrak{S}_d$ by 
$$D_{\tau,i}:=\bigoplus_{j=1}^iK_0\otimes_{\mathbb{Q}_p}Ee_{\tau(j)}$$ 
for $1\leqq i\leqq d$, whose Hodge filtrations are 
induced from that on $\bold{D}_{\mathrm{cris}}(V)$, i.e. we define $\mathrm{Fil}^k
(K\otimes_{K_0}D_{\tau,i}):=(K\otimes_{K_0}D_{\tau,i})\cap \mathrm{Fil}^k(K\otimes_{K_0}\bold{D}_{\mathrm{cris}}(V))$ for each $k\in \mathbb{Z}$. By the equivalence between the category 
of $E$-filtered $\varphi$-modules and the category of crystalline $E$-$B$-pairs (see \cite{Be08} or \cite{Na09}, \cite{Na10}), for each $\tau\in\mathfrak{S}_d$,
we obtain a triangulation
$$T_{\tau}:0\subseteq W_{\tau,1}\subseteq W_{\tau,2}\subseteq \cdots\subseteq W_{\tau,d}=W(V)$$ 
by crystalline $E$-$B$-pairs $\{W_{\tau,i}\}_{1\leqq i\leqq d}$ such that $\bold{D}_{\mathrm{cris}}(W_{\tau,i})\isom D_{\tau,i}$ for any $1\leqq i\leqq d$. We recall the definition of benign representation in \cite{Na10} which is also called generic crystalline representation in \cite{Ch11}, whose deformation theoretic property plays a crucial role in 
the proof of the main theorem of this article. Let $\{k_{1,\sigma},k_{2,\sigma},\cdots, k_{d,\sigma}\}_{\sigma\in \mathcal{P}}$ be 
the set of Hodge-Tate weights of $V$ (possibly with multiplicity) such that $k_{1,\sigma}\geqq k_{2,\sigma}\geqq \cdots\geqq k_{d,\sigma}$ for any 
$\sigma\in \mathcal{P}$. In this article, we define the Hodge-Tate weight of the $p$-adic cyclotomic character 
$\chi_p:G_K\rightarrow E^{\times}$ to be  $\{1\}_{\sigma\in\mathcal{P}}$.

\begin{defn}
Let $V$ be a crystalline representation satisfying all the assumptions in the previous paragraph. 
We say that $V$ is benign if $V$ satisfies the following conditions,
\begin{itemize}
\item[(1)]$\alpha_i\not=\alpha_j, p^{\pm f}\alpha_j$ for any $i\not=j$,
\item[(2)]$k_{1,\sigma}>k_{2,\sigma}>\cdots>k_{d,\sigma}$ for any $\sigma\in \mathcal{P}$,
\item[(3)]the Hodge-Tate weights of 
$W_{\tau,i}$ is \allowbreak $\{k_{1,\sigma},k_{2,\sigma},\cdots,k_{i,\sigma}\}_{\sigma\in \mathcal{P}}$ 
for each $\tau\in \mathfrak{S}_d$ and $1\leqq i\leqq d$.

\end{itemize}

\end{defn}

If $V$ is benign, then the pair $(V,T_{\tau})$ satisfies all the conditions in Proposition \ref{4.2} for any $\tau\in \mathfrak{S}_d$, hence 
the functors $D_V$ and $D_{V,T_{\tau}}$ for all $\tau\in\mathfrak{S}_d$ 
are  representable by $R_{V}$ and $R_{V,T_{\tau}}$ which are quotients 
of $R_V$. For $R_{*}=R_V, R_{V,T_{\tau}}$, we define the tangent space of $R_{*}$ by 
$$t_{R_{*}}:=\mathrm{Hom}_E(\mathfrak{m}_{R_{*}}/\mathfrak{m}^2_{R_{*}}, E),$$ 
where $\mathfrak{m}_{R_{*}}$ is the maximal ideal of $R_*$. Hence, we obtain a natural 
inclusion $t_{R_{V,T_{\tau}}}\hookrightarrow t_{R_V}$ for each $\tau\in \mathfrak{S}_d$.

The following theorem is a crucial for the proof of the main theorem of this article, which 
was discovered by Chenevier (Theorem 3.19 of \cite{Ch11}).

\begin{thm}\label{4.3}
Let $V$ be a benign representation of rank $d$, then we have an equality 
$$\sum_{\tau\in \mathfrak{S}_d}t_{R_{V,T_{\tau}}}=t_{R_V}.$$

\end{thm}
\begin{proof}
See  Theorem 3.19 of \cite{Ch11} (for $K=\mathbb{Q}_p$) and Theorem 2.61 of \cite{Na10} 
(for general $K$).

\end{proof}

\section{Construction of the local eigenvarieties}
This section is the technical heart of this article. We construct some (approximation of) 
$p$-adic families of trianguline representations. In $\S10$ of \cite{Ki03} (for $K=\mathbb{Q}_p$) 
and $\S3$ of \cite{Na10} (for general $K$), they constructed such families for two dimensional case using the theory of finite slope subspace. In this section, we generalize their constructions for higher 
dimensional case by using the (slightly generalized version of) finite slope subspace and a technique of \cite{BeCh09} for the study of trianguline deformations using exterior products.

\subsection{finite slope subspace}
We first generalize Proposition 5.4 of \cite{Ki03} and Theorem 3.9 of \cite{Na10} as follows. 
We first recall some terminologies which are used in \cite{Ki03},\cite{Na10} (see \S2 and \S 3 of \cite{Na10} for more details). 

Let $\mathbb{C}_p$ be the $p$-adic completion of an algebraic closure $\overline{\mathbb{Q}}_p$ of $\mathbb{Q}_p$. Denote by $v_p:\mathbb{C}_p^{\times}\rightarrow \mathbb{Q}$ the valuation such that 
$v_p(p)=1$. Denote by $\widetilde{\bold{E}}^+:=\varprojlim_{n\geqq 0}\mathcal{O}_{\mathbb{C}_p}/p$ the projective limit by the $p$-th power map. Define the valuation $v$ on $\widetilde{\bold{E}}^+$ by 
$v((\overline{x}_n)_{n\geqq 0}):=\lim_{n\rightarrow \infty}v_p(x_n)$ where $x_n\in \mathcal{O}_{\mathbb{C}_p}$ is a lift of $\bar{x}_n$. Take a set $\{p_n\}_{n\geqq 0}\subseteq \overline{\mathbb{Q}}_p$ such that $p_0=p$, $p_{n+1}^p=p_n$ for any $n\geqq 0$. 
Set $\tilde{p}:=(\bar{p}_n)_{n\geqq 0}\in \widetilde{\bold{E}}^+$. Denote by $W(\widetilde{\bold{E}}^+)$ 
the ring of Witt vectors of $\widetilde{\bold{E}}^+$, and by $[-]:\widetilde{\bold{E}}^+\rightarrow 
W(\widetilde{\bold{E}}^+)$ the Teichm\"uler lift. Set $\bold{A}_{\mathrm{max}}:=W(\widetilde{\bold{E}}^+)[\frac{[\tilde{p}]}{p}]^{\wedge}$, where $(-)^{\wedge}$ is the $p$-adic completion of $(-)$. The actions of $G_K$ and Frobenius 
$\varphi$ on $W(\widetilde{\bold{E}}^+)$ naturally extend to $\bold{A}_{\mathrm{max}}$ and 
$\bold{B}^+_{\mathrm{max}}:=\bold{A}_{\mathrm{max}}[1/p]$. Define a $K$-Banach algebra $\bold{B}^+_{\mathrm{max},K}
:=K\otimes_{K_0}\bold{B}^+_{\mathrm{max}}$,  and define $\varphi_K:=\mathrm{id}_K\otimes\varphi^f
:\bold{B}^+_{\mathrm{max},K}\rightarrow \bold{B}^+_{\mathrm{max},K}$.

Let $X$ be a separated rigid analytic space over $E$ in the sense of Tate. 
For $x\in X$, we denote by $E(x)$ the residue field of $X$ at $x$, which is 
a finite extension of $E$. We say that an admissible open set $U\subseteq X$ is scheme theoretically dense 
in $X$ if there exists an admissible affinoid covering $\{X_i:=\mathrm{Spm}(R_i)\}_{i\in I}$ of $X$ such that 
$U\cap X_i$ is associated to a dense Zariski open $U_i\subseteq \mathrm{Spec}(R_i)$ for any $i\in I$.
For an invertible function $Y\in\Gamma(X, \mathcal{O}_X)^{\times}$ and an $E$-affinoid algebra $R$, we say that 
an $E$-morphism $f:\mathrm{Spm}(R)\rightarrow X$ is $Y$-small if there exist a finite extension $E'$ of $E$ 
and an element $\lambda\in (R\otimes_E E')^{\times}$ such that $E'[\lambda]\subseteq R\otimes_EE'$ is a finite \'etale 
$E'$-algebra and $\frac{Y}{\lambda}$ is topologically nilpotent in $R\otimes_E E'$. For any $f\in \Gamma(X,\mathcal{O}_X)$, 
we denote by $X_f:=\{x\in X| f(x)\not=0\}$ the Zariski open of $X$ on which $f$ is not zero.

For a finite free $\mathcal{O}_X$-module $M$ with a continuous $\mathcal{O}_X$-linear $G_K$-action, 
we denote by $M(x)$ the fiber of $M$ at $x$, which is an $E(x)$-representation of $G_K$. We denote by $M^{\vee}$ the $\mathcal{O}_X$-dual of $M$.
For such $M$ of rank $n$, 
we can define Sen's polynomial $$P_{M}(T)\in K\otimes_{\mathbb{Q}_p}\Gamma(X, \mathcal{O}_X)[T],$$ 
which is a monic polynomial of degree $n$, 
such that the fiber $P_{M}(T)(x)$ at $x$ is equal to Sen's polynomial $P_{M(x)}(T)\in 
K\otimes_{\mathbb{Q}_p}E(x)[T]$ of $M(x)$ for any $x\in X$(see for example \cite{Ki03} (2.2)). Using the canonical 
decomposition 
$$K\otimes_{\mathbb{Q}_p}\Gamma(X,\mathcal{O}_X)[T]\isom \prod_{\sigma\in \mathcal{P}}\Gamma(X,\mathcal{O}_X)[T]: 
a\otimes f(T)\mapsto (\sigma(a)f(T))_{\sigma\in \mathcal{P}},$$ we decompose $P_{M}(T)$ into
$$P_{M}(T)=(P_{M,\sigma}(T))_{\sigma\in \mathcal{P}}\in \prod_{\sigma\in \mathcal{P}}\Gamma(X,\mathcal{O}_X)[T].$$ 
Let $d\in \mathbb{Z}_{\geqq 1}$ be a positive integer. Assume that we are given $d$-pairs $\{(M_i,Y_i)\}_{1\leqq i\leqq d}$, where $M_i$ are finite free $\mathcal{O}_X$-modules with continuous $\mathcal{O}_X$-linear $G_K$-actions and $Y_i\in \Gamma(X,\mathcal{O}_X)^{\times}$.
We assume that the $\sigma$-part of Sen's polynomial $P_{M_i,\sigma}(T)$ of $M_i$ can be written as 
$$P_{M_i,\sigma}(T)=TQ_{i,\sigma}(T)$$ for a monic polynomial 
$Q_{i,\sigma}(T)\in \Gamma(X, \mathcal{O}_X)[T]$ for any $1\leqq i\leqq d$ and $\sigma\in \mathcal{P}$.

Under this situation, we prove the following theorem, which is a slightly  generalized version of Proposition 
5.4 of \cite{Ki03} and Theorem 3.9 of \cite{Na10}.

\begin{thm}\label{1.1}
Under the above situation, there exists a unique Zariski closed subspace $X_{fs}\subseteq X$ 
satisfying the following conditions (1) and (2).
\begin{itemize}
\item[(1)]For any $1\leqq i\leqq d$, $\sigma\in \mathcal{P}$ and 
$j\in \mathbb{Z}_{\leqq 0}$, the subset $X_{fs, Q_{i,\sigma}(j)}$ is scheme theoretically 
dense in $X_{fs}$.
\item[(2)]For any $E$-morphism $f:\mathrm{Spm}(R)\rightarrow X$ which 
is $Y_i$-small for any $1\leqq i\leqq d$ and factors through $X_{Q_{i,\sigma}(j)}$ for any 
$1\leqq i\leqq d$, $\sigma\in \mathcal{P}$ and $j\in \mathbb{Z}_{\leqq 0}$, the following 
conditions (i) and (ii) are equivalent.
\begin{itemize}
\item[(i)]$f$ factors through $f:\mathrm{Spm}(R)\rightarrow X_{fs}\hookrightarrow X$.
\item[(ii)]For any $1\leqq i\leqq d$, any $R$-linear $G_K$-equivariant map 
$$h:f^*(M_i^{\vee})\rightarrow \bold{B}^+_{\mathrm{dR}}\hat{\otimes}_{\mathbb{Q}_p}R$$ factors through 
$$h:f^*(M_i^{\vee})\rightarrow K\otimes_{K_0}(\bold{B}^+_{\mathrm{max}}\hat{\otimes}_{\mathbb{Q}_p}R)^{\varphi^f=Y_i}
\hookrightarrow \bold{B}^+_{\mathrm{dR}}\hat{\otimes}_{\mathbb{Q}_p}R.$$
\end{itemize}

\end{itemize}

\end{thm}
\begin{proof}
The proof of the uniqueness is the same as that of Proposition 5.4 of \cite{Ki03} or Theorem 3.9 
of \cite{Na10}. 

By the same argument as in \cite{Ki03} or \cite{Na10}, it suffices to construct 
$X_{fs}$ when $X=\mathrm{Spm}(R)$ is an affinoid $E$-algebra which satisfies  $|Y_i||Y_i^{-1}|< \frac{1}{|\pi_K|_p}$ for any 
$1\leqq i\leqq d$, where $| - |:R\rightarrow \mathbb{Q}_{\geqq 0}$ is an $E$-Banach norm on $R$. Then, we construct $X_{fs}$ as follows. 

Take a sufficiently large $k\in\mathbb{Z}_{\geqq1}$ such that, for any $1\leqq i \leqq d$ and $\lambda_i 
\in \overline{E}$ such that $|Y_i^{-1}|^{-1}\leqq |\lambda_i|_p\leqq |Y_i|$, the natural map 
$$(\bold{B}^+_{\mathrm{max},K}\otimes_{K,\sigma}E')^{\varphi_K=\sigma(\pi_K)\lambda_i}\hookrightarrow 
\bold{B}^+_{\mathrm{dR}}/t^k\bold{B}^+_{\mathrm{dR}} \otimes_{K,\sigma}E'$$ is injective with a closed image for any $\sigma\in \mathcal{P}$, which is 
possible 
by Corollary 3.5 of \cite{Na10}.
 Let $U_{i,\sigma}$ be the cokernel of this map, then 
$U_{i,\sigma}$ is also an $E'$-Banach space and we fix an orthonormalizable basis $\{e_{i,\sigma,j}\}_{j\in J_{i,\sigma}}$ of $U_{i,\sigma}$. Then, for any $R$-linear $G_K$-morphism $$h:M_i^{\vee}\rightarrow \bold{B}^+_{\mathrm{dR}}/t^k\bold{B}^+_{\mathrm{dR}}\hat{\otimes}_{K,\sigma}R$$
and $x\in \widetilde{\bold{E}}^+$ such that $v(x)>0$, we denote by 
$$h_x:M_i^{\vee}\rightarrow U_{i,\sigma}\hat{\otimes}_{E'}(R\otimes_E E')$$ the composition of 
$h$ with the map $$\bold{B}^+_{\mathrm{dR}}/t^k\bold{B}^+_{\mathrm{dR}}\hat{\otimes}_{K,\sigma}R\rightarrow \bold{B}^+_{\mathrm{dR}}/t^k\bold{B}^+_{\mathrm{dR}}\hat{\otimes}_{K,\sigma}(R\otimes_EE'):
y\mapsto P\Bigl(x, \frac{Y_i}{\sigma(\pi_K)\lambda_i}\Bigr)y$$ and the natural quotient map 
$$\bold{B}^+_{\mathrm{dR}}/t^k\bold{B}^+_{\mathrm{dR}}
\hat{\otimes}_{K,\sigma}(R\otimes_EE')\twoheadrightarrow U_{i,\sigma}\hat{\otimes}_{E'}(R\otimes_E E'),$$ where $P\Bigl(x,\frac{Y_i}{\sigma(\pi_K)\lambda_i}\Bigr)$ is defined by 
$$P\Bigl(x,\frac{Y_i}{\sigma(\pi_K)\lambda_i}\Bigr):=
\sum_{n\in \mathbb{Z}} \varphi_K([x])^n\otimes \Bigl(\frac{Y_i}{\sigma(\pi_K)\lambda_i}\Bigr)^n\in (\bold{B}^+_{\mathrm{max},K}\hat{\otimes}_{K,\sigma}(R\otimes_E E'))^{\varphi_K=\frac{\sigma(\pi_K)\lambda_i}{Y_i}}$$  whose convergence is proved in the proof of Theorem 3.9 of \cite{Na10}. 
Then, for any $m\in M_i^{\vee}$, 
we can write uniquely 
$$h_x(m)=\sum_{j\in J_{i,\sigma}}a(h,x,\lambda_i,m)_je_{i,\sigma,j}$$ for some $\{a(h,x,\lambda_i,m)_j\}_{j\in J_{i,\sigma}}\subseteq  R\otimes_EE'$. We define an ideal 
$$I(h,x,\lambda_i)'\subseteq R\otimes_E E'$$ which is generated by $a(h,x,\lambda_i,m)_j$ for all $m\in M_i^{\vee}$ and $j\in J_{\sigma}$.
Because we have an equality $I(h,x,\tau(\lambda_i))'=\tau(I(h,x,\lambda_i)')$ for any $\tau\in \mathrm{Gal}(E'/E)$, the ideal 
$\sum_{\tau\in \mathrm{Gal}(E'/E)}I(h,x,\tau(\lambda_i))'$ descends to an ideal 
$$I(h,x,\lambda_i)\subseteq R.$$ We define an ideal by
$$I:=\sum_{i,h,x,\lambda_i}I(h,x,\lambda_i)\subseteq R.$$ 
Finally, we define the smallest ideal $I'$ so that $I'$ contains $I$ and the natural map $R/I'\rightarrow 
R/I'[\frac{1}{Q_{i,\sigma}(j)}]$ is an injection for any $1\leqq i\leqq d$, $\sigma\in \mathcal{P}$ and 
$j\in \mathbb{Z}_{\leqq 0}$. Then, the closed subspace $\mathrm{Spm}(R/I')$ satisfies the conditions (1) and (2), which we can 
prove in the same way as in the proof of Proposition 5.4 of \cite{Ki03} or Theorem 3.9 of \cite{Na10}.

\end{proof}
Next, we prove a proposition concerning some important properties of $X_{fs}$, which is a generalization of 
Proposition 5.14 of \cite{Ki03} and Proposition 3.14 of \cite{Na10}. 
Let $U=\mathrm{Spm}(R)$ be an affinoid open of $X_{fs}$ which 
is $Y_i$-small for any $1\leqq i\leqq d$. By Proposition 3.7 of \cite{Na10}, 
 for any sufficiently large $k\in \mathbb{Z}_{\geqq 1}$, there exists a short exact sequence,
$$0\rightarrow (\bold{B}^+_{\mathrm{max},K}\hat{\otimes}_{K,\sigma}R)^{\varphi_K=Y_i}
\rightarrow \bold{B}^+_{\mathrm{dR}}/t^k\bold{B}^+_{\mathrm{dR}}\hat{\otimes}_{K,\sigma}R \rightarrow U_{i,\sigma}
\rightarrow 0$$ for each
$\sigma\in \mathcal{P}$ and $1\leqq i\leqq d$, where 
$U_{i,\sigma}$ is a Banach $R$-module which is a direct summand of an orthonormalizable 
Banach $R$-module.
\begin{prop}\label{1.2}
In the above situation, the following hold.
\begin{itemize}
\item[(1)]For any $1\leqq i\leqq d$ and $\sigma\in \mathcal{P}$, the natural injection 
$$(\bold{B}^+_{\mathrm{max},K}\hat{\otimes}_{K,\sigma}(M_i\otimes_{\mathcal{O}_X}R))^{G_K, \varphi_K=Y_i}\hookrightarrow 
(\bold{B}^+_{\mathrm{dR}}/t^k\bold{B}^+_{\mathrm{dR}}\hat{\otimes}_{K,\sigma} (M_i\otimes_{\mathcal{O}_X}R))^{G_K}$$
is an isomorphism.
\item[(2)]For $1\leqq i\leqq d$ and $\sigma\in \mathcal{P}$, let $H_{i,\sigma}\subseteq R$ be the smallest ideal such that 
any $G_K$-equivariant $R$-linear map $$h:M_i^{\vee}\rightarrow \bold{B}^+_{\mathrm{dR}}/t^k\bold{B}^+_{\mathrm{dR}}\hat{\otimes}_{K,\sigma}
R$$ factors through $$\bold{B}^+_{\mathrm{dR}}/t^k\bold{B}^+_{\mathrm{dR}}\hat{\otimes}_{K,\sigma} H_{i,\sigma}\hookrightarrow 
\bold{B}^+_{\mathrm{dR}}/t^k\bold{B}^+_{\mathrm{dR}}\hat{\otimes}_{K,\sigma}R.$$ Set $H:=\prod_{1\leqq d,\sigma\in \mathcal{P}}H_{i,\sigma}\subseteq R$. Then $\mathrm{Spm}(R)\setminus V(H)$ and $\mathrm{Spm}(R)\setminus V(H_{i,\sigma})$ are scheme theoretically 
dense in $\mathrm{Spm}(R)$.
\item[(3)]For any $x\in \mathrm{Spm}(R)$, 
$(\bold{B}^+_{\mathrm{max},K}\otimes_{K,\sigma}M_i(x))^{G_K,\varphi_K=Y_i}$ is non zero for any $1\leqq i\leqq d$ and $\sigma\in \mathcal{P}$.
\end{itemize}

\end{prop}
\begin{proof}
The proof is essentially  the same as that of Proposition 3.14 of  \cite{Na10}.
First, we prove (1).  By the definition of $X_{fs}$, we have an equality 
$$(\bold{B}^+_{\mathrm{max},K}\otimes_{K,\sigma}(M_i\otimes_{\mathcal{O}_X} R_{x}/\mathfrak{m}_x^n))^{G_K, 
\varphi_K=Y_i}=(\bold{B}^+_{\mathrm{dR}}/t^k\bold{B}^+_{\mathrm{dR}}\otimes_{K,\sigma}(M_i\otimes_{\mathcal{O}_X}R_x/\mathfrak{m}_x^n))^{G_K}$$ for 
any $x\in \mathrm{Spm}(R)$ and $n\geqq 1$ such that $Q_{i,\sigma}(j)(x)\not=0$ for any $\sigma\in \mathcal{P}$, $1\leqq i\leqq d$ and 
$j\in \mathbb{Z}_{\leqq 0}$, where $R_x$ is the local ring at $x$. Hence, it suffices to show that the natural map 
$R\rightarrow \prod_{x\in V, n\geqq 1}R_x/\mathfrak{m}_x^n$ is an injection, where we define
$$V:=\{x\in \mathrm{Spm}(R)| 
Q_{i,\sigma}(j)(x)\not=0 \text{ for any } \sigma\in\mathcal{P}, 1\leqq i\leqq d, j\in \mathbb{Z}_{\leqq 0} \}.$$ Let $f\in R$ be an element in the kernel of this map. For $Q\in R$, we denote by $V(Q)$ the reduced closed subspace of $\mathrm{Spm}(R)$ such that $V(Q)=\{x\in \mathrm{Spm}(R)| Q(x)=0\}$. 
If we denote by $W$ the support of $f$ in $\mathrm{Spm}(R)$, then we have an inclusion 
$W\subseteq  \cup_{\sigma\in \mathcal{P}, 
1\leqq i\leqq d, j\in\mathbb{Z}_{\leqq 0}}V(Q_{i,\sigma}(j))$. 
By Lemma 5.7 of \cite{Ki03}, there exists $Q$ which is a finite product of 
$Q_{i,\sigma}(j)$ such that $W\subseteq V(Q)$, hence we obtain an inclusion 
$\mathrm{Spm}(R)_{Q}\subseteq X\setminus W$, in particular 
$f$ is zero in $R[\frac{1}{Q}]$. Because $\mathrm{Spm}(R)_Q$ is scheme theoretically dense in $\mathrm{Spm}(R)$,  we have $f=0$, 
which proves (1).

Next, we prove (2). We first show that if $x\in V$ then $x\in \mathrm{Spm}(R)\setminus V(H_{i,\sigma})$ for any $1\leqq i\leqq d$ and $\sigma\in \mathcal{P}$ and $x\in \mathrm{Spm}(R)\setminus V(H)$. 
If $x\in V\cap V(H_{i,\sigma})$ for some $1\leqq i\leqq d$ and $\sigma\in \mathcal{P}$, then we have an equality 
$$(\bold{B}^+_{\mathrm{dR}}/t^k\bold{B}^+_{\mathrm{dR}}\otimes_{K,\sigma}
(M_i\otimes_{\mathcal{O}_X}R))^{G_K}\otimes_R R_x/\mathfrak{m}_x=(\bold{B}^+_{\mathrm{dR}}/t^k\bold{B}^+_{\mathrm{dR}}\otimes_{K,\sigma}
(M_i\otimes_{\mathcal{O}_X}R_x/\mathfrak{m}_x))^{G_K}$$ which is a one dimensional $E(x)$-vector space by Corollary 2.6 of \cite{Ki03}. However, the left hand side is zero because we have 
$H_{i,\sigma}\subseteq \mathfrak{m}_x$, and this is a contradiction. 
Hence, by the same argument as in the proof of (1), there exists $Q$ which is a finite product of $Q_{i,\sigma}(j)$ such that 
$\mathrm{Spm}(R)_Q\subseteq \mathrm{Spm}(R)\setminus V(H)$. Because $\mathrm{Spm}(R)_Q$ is 
scheme theoretically dense  in $\mathrm{Spm}(R)$ by the definition of $X_{fs}$, this inclusion implies that $\mathrm{Spm}(R)\setminus V(H)$ is also scheme 
theoretically dense in $\mathrm{Spm}(R)$, and then $\mathrm{Spm}(R)\setminus V(H_{i,\sigma})$ are also scheme theoretically dense for all $i,\sigma$.

Using (2), we can prove (3) in the same way as that of Proposition 3.14 of \cite{Na10}. 
\end{proof}

\subsection{construction of the local eigenvariety}
As a generalization of $\S10$ of \cite{Ki03} and $\S3$ of \cite{Na10} to the 
higher dimensional case, we apply Theorem \ref{1.1} to the following situation.
Let $\overline{V}$ be an $\mathbb{F}$-representation 
of $G_K$ of dimension $d$. Let $\mathcal{C}_{\mathcal{O}}$ be the category of Artin local 
$\mathcal{O}$-algebra $A$ with residue field $\mathbb{F}$. 
For $A\in \mathcal{C}_{\mathcal{O}}$, we say that a couple $(V_A,\psi_A)$ is a deformation of $\overline{V}$ over $A$
if $V_A$ is an $A$-representation of $G_K$ and $\psi_A: V_A\otimes_A \mathbb{F}\isom \overline{V}$ is 
an isomorphism of $\mathbb{F}$-representations. We say that two deformations $(V_A,\psi_A)$ 
and $(V'_A,\psi'_A)$ of $\overline{V}$ over $A$ are equivalent if there exists an isomorphism 
$f:V_A\isom V'_A$ of $A$-representations such that $\psi_A=\psi'_A\circ (f\otimes\mathrm{id}_{\mathbb{F}})$.
We consider a functor $$D_{\overline{V}}:\mathcal{C}_{\mathcal{O}}\rightarrow (Sets)$$ defined 
by $$D_{\overline{V}}(A):=\{\text{ equivalent classes of deformations of }\overline{V} \text{ over } A\}$$ 
 for $A\in \mathcal{C}_{\mathcal{O}}$. We simply write $[V_A]\in D_{\overline{V}}(A)$ instead of $[(V_A,\psi_A)]$ if there is no risk of confusing about $\psi_A$.
In this paper, for simplicity, we assume that $\overline{V}$ satisfies 
$$\mathrm{End}_{\mathbb{F}[G_K]}(\overline{V})=\mathbb{F}.$$ 
Under this assumption, the functor $D_{\overline{V}}$ is pro-representable by the universal deformation 
ring $R_{\overline{V}}$. Let $V^{\mathrm{univ}}$ be the universal deformation of $\overline{V}$ over $R_{\overline{V}}$.
Let $\mathcal{X}_{\overline{V}}$ be the rigid analytic space over $E$ associated to the formal $\mathcal{O}$-scheme $\mathrm{Spf}(R_{\overline{V}})$. 
Then, $V^{\mathrm{univ}}$ naturally defines a finite free $\mathcal{O}_{\mathcal{X}_{\overline{V}}}$-module $M$ of rank $d$ with 
a continuous $\mathcal{O}_{\mathcal{X}_{\overline{V}}}$-linear $G_K$-action. 
Let 
$$P_{M}(T):=(P_{M,\sigma}(T))_{\sigma\in \mathcal{P}}
\in K\otimes_{\mathbb{Q}_p}\mathcal{O}_{\mathcal{X}_{\overline{V}}}[T]\isom \oplus_{\sigma\in \mathcal{P}}\mathcal{O}_{\mathcal{X}_{\overline{V}}}[T]$$ be  Sen's polynomial of $M$.

Let $\mathrm{det}(M):G_K^{\mathrm{ab}}\rightarrow \mathcal{O}_{\mathcal{X}_{\overline{V}}}^{\times}$ be the determinant character 
of $M$. 
We also write by the same letter $\mathrm{det}(M):K^{\times}\rightarrow 
\mathcal{O}_{\mathcal{X}_{\overline{V}}}^{\times}$ the continuous homomorphism defined as $\mathrm{det}(M)\circ \mathrm{rec}_K$.

We next recall the definitions of the weight spaces for $\mathcal{O}_K^{\times}$ and $K^{\times}$. Let $\mathcal{W}$ and $\mathcal{T}$ be 
the functors from the category of rigid analytic spaces over $E$ to the category of abelian groups defined 
by
$$\mathcal{W}(X'):=\{\delta:\mathcal{O}_K^{\times}\rightarrow \Gamma(X', \mathcal{O}_{X'})^{\times}
|\delta \text{ is a continuous homomorphism} \}$$ and $$\mathcal{T}(X'):=\{\delta:K^{\times}\rightarrow \Gamma(X', \mathcal{O}_{X'})^{\times}| 
\delta\text{ is a continuous homomorphism}\}$$ 
for each rigid analytic space $X'$ over $E$. It is known that $\mathcal{W}$ and $\mathcal{T}$ are representable and 
$\mathcal{W}$ is representable
by the rigid analytic group variety associated to the Iwasawa algebra $\mathcal{O}[[\mathcal{O}_K^{\times}]]$. As a rigid space over $E$, $\mathcal{W}$ is non-canonically isomorphic to 
$\sharp(\mathcal{O}^{\times}_{K, \mathrm{tor}})$-union of 
$d$-dimensional open unit discs. If we fix a uniformizer $\pi_K\in \mathcal{O}_K$, we have an 
isomorphism $$\mathcal{T}\isom \mathcal{W}\times_E \mathbb{G}_{m/E}^{\mathrm{an}}:\delta\mapsto ( \delta|_{\mathcal{O}_K^{\times}},\delta(\pi_K)).$$ We denote the projections by $$p_{1}:\mathcal{T}\rightarrow \mathcal{W}:\delta\mapsto \delta|_{\mathcal{O}_K^{\times}}, 
\,p_{2,\pi_K}:\mathcal{T}\rightarrow \mathbb{G}_{m/E}^{\mathrm{an}}:\delta\mapsto \delta(\pi_K).$$ Let $Y\in \Gamma(\mathbb{G}_{m/E}^{\mathrm{an}},\mathcal{O}_{\mathbb{G}_{m/E}^{\mathrm{an}}})^{\times}$ be the canonical coordinate.
Let $$\delta_{\mathcal{W}}^{\mathrm{univ}}:\mathcal{O}_K^{\times}\rightarrow \Gamma(\mathcal{W}, \mathcal{O}_{\mathcal{W}})^{\times}$$ be 
the universal homomorphism of the functor $\mathcal{W}$, which is equal to the composite of the canonical map $\mathcal{O}_K^{\times}\rightarrow \mathcal{O}[[\mathcal{O}_K^{\times}]]^{\times}: a\mapsto [a]$ with the canonical map 
$\mathcal{O}[[\mathcal{O}_K^{\times}]]^{\times}\rightarrow \Gamma(\mathcal{W}, \mathcal{O}_{\mathcal{W}})^{\times}$.
Let 
$$\widetilde{\delta}_{\mathcal{W}}^{\mathrm{univ}}:G_K^{\mathrm{ab}}\rightarrow \Gamma(\mathcal{W},\mathcal{O}_{\mathcal{W}})^{\times}$$ be the continuous character defined by 
$$\widetilde{\delta}_{\mathcal{W}}^{\mathrm{univ}}\circ \mathrm{rec}_K|_{\mathcal{O}_K^{\times}}=\delta_{\mathcal{W}}^{\mathrm{univ}}
\text{ and  }\widetilde{\delta}^{\mathrm{univ}}_{\mathcal{W}}(\mathrm{rec}_K(\pi_K))=1.$$
Then the universal homomorphism $$\delta^{\mathrm{univ}}_{\mathcal{T}}:K^{\times}\rightarrow \Gamma(\mathcal{T},\mathcal{O}_{\mathcal{T}})^{\times}$$ of the functor $\mathcal{T}$ satisfies 
$$\delta^{\mathrm{univ}}_{\mathcal{T}}|_{\mathcal{O}_K^{\times}}=
p^*_1\circ\delta^{\mathrm{univ}}_{\mathcal{W}}\text{ and } \delta^{\mathrm{univ}}_{\mathcal{T}}(\pi_K)=p_{2,\pi_K}^*(Y).$$

We define the notion of the generalized Hodge-Tate weights for any $\delta\in \mathcal{W}(X')$. 
For any rigid space $X'$ over $E$, any continuous homomorphism $\delta:\mathcal{O}_K^{\times}
\rightarrow \Gamma(X',\mathcal{O}_{X'})^{\times}$ is locally $\mathbb{Q}_p$-analytic by Proposition 8.3 of \cite{Bu07}, i.e. 
locally around $1\in\mathcal{O}^{\times}_K$, $\delta$ can be written by 
$$\delta(x)=\sum_{\bold{n}=\{n_{\sigma}\}_{\sigma\in\mathcal{P}},n_{\sigma}\geqq 0}a_{\bold{n}}\prod_{\sigma\in\mathcal{P}}\sigma(x-1)^{n_{\sigma}}$$
for a unique $\{a_{\bold{n}}\}_{\bold{n}}\subseteq \Gamma(X',\mathcal{O}_{X'})$. Then, for any 
$\sigma\in \mathcal{P}$, we define the $\sigma$-part of the generalized Hodge-Tate wight $k(\delta)_{\sigma}$ of $\delta$ as the partial 
differential $\frac{\partial\delta(x)}{\partial \sigma(x)}|_{x=1}$ of $\delta$ by $\sigma(x)$ at $1$, more precisely, we define 
$$k(\delta)_{\sigma}:=a_{\bold{n}}\text{ for } \bold{n}=\{n_{\sigma'}\}_{\sigma'\in\mathcal{P}} \text{ such that } n_{\sigma}=1\text{ and }
n_{\sigma'}=0  \,\,(\sigma'\not=\sigma).$$
 
Here, we prove a proposition concerning the generalized Hodge-Tate weights of rank one 
representations of $G_K$, which justifies the above definition. We recall that  $\chi_{\mathrm{LT}}:G_K^{\mathrm{ab}}\rightarrow \mathcal{O}_K^{\times}$ is 
the Lubin-Tate character associated to a fixed uniformizer $\pi_K\in \mathcal{O}_K$. 




\begin{prop}\label{a-3}
Let $X'$ be a rigid space over $E$, and let $\widetilde{\delta}:G_K^{\mathrm{ab}}\rightarrow \Gamma(X',\mathcal{O}_{X'})^{\times}$ be a continuous 
character. Set $\delta:=\widetilde{\delta}\circ \mathrm{rec}_K|_{\mathcal{O}_K^{\times}}\in \mathcal{W}(X')$. 
Then, the $\sigma$-part of the generalized Hodge-Tate weight of $\mathcal{O}_{X'}(\widetilde{\delta})$ is equal to 
$k(\delta)_{\sigma}$. 
\end{prop}
\begin{proof}
Let $\widetilde{\delta}:G_K^{\mathrm{ab}}\rightarrow \Gamma(X',\mathcal{O}_{X'})^{\times}$ be a continuous character. Because 
twisting by a unramified character does not  change the Hodge-Tate weights, we may assume that 
$\widetilde{\delta}(\mathrm{rec}_K(\pi_K))=1$. By the universality of $\mathcal{W}$, there exists a morphism 
$f:\mathrm{Spm}(A)\rightarrow \mathcal{W}$ such that $\delta=f^*\circ\delta^{\mathrm{univ}}_{\mathcal{W}}$, which also implies the equality $\widetilde{\delta}=f^{*}\circ\widetilde{\delta}^{\mathrm{univ}}_{\mathcal{W}}$.
Because both $k(\delta)_{\sigma}$ and the Hodge-Tate weights of $\widetilde{\delta}$ are compatible with 
base changes, it suffices to show that the $\sigma$-part of 
the generalized Hodge-Tate weight of $\widetilde{\delta}_{\mathcal{W}}^{\mathrm{univ}}$ is equal to 
$k(\delta^{\mathrm{univ}}_{\mathcal{W}})_{\sigma}$ for each $\sigma\in\mathcal{P}$. 
We denote by $a_{\sigma}\in \Gamma(\mathcal{W}, \mathcal{O}_{\mathcal{W}})$ 
the $\sigma$-part of the generalized Hodge-Tate weight of 
$\widetilde{\delta}_{\mathcal{W}}^{\mathrm{univ}}$. Define a subset 
$$\mathcal{W}_{0}:=\{\prod_{\sigma\in \mathcal{P}}\sigma^{k_{\sigma}}:\mathcal{O}_K^{\times}
\rightarrow E^{\times}| k_{\sigma} \in \mathbb{Z} \text{ for any }\sigma\in\mathcal{P}\}\subseteq \mathcal{W}.$$
 Then
$a_{\sigma}$ is equal to $k(\delta^{\mathrm{univ}}_{\mathcal{W}})_{\sigma}$ at any points of $\mathcal{W}_0$ because the character $\widetilde{(\prod_{\sigma\in\mathcal{P}}\sigma^{k_{\sigma}})}
:G_K^{\mathrm{ab}}\rightarrow E^{\times}$ which is equal to $\prod_{\sigma\in \mathcal{P}}\sigma(\chi_{\mathrm{LT}})^{k_{\sigma}}$ is a crystalline character 
with the generalized Hodge-Tate weights $\{k_{\sigma}\}_{\sigma\in \mathcal{P}}$ by the result of Fontaine. Since the subset $\mathcal{W}_0$ is Zariski 
dense in $\mathcal{W}$ by Lemma 2.7 of \cite{Ch09}, $a_{\sigma}$ is equal to $k(\delta^{\mathrm{univ}}_{\mathcal{W}})_{\sigma}$ on $\mathcal{W}$, which proves the proposition.

\end{proof}

Set $\mathcal{Z}:=\mathcal{X}_{\overline{V}}\times_E\mathcal{T}^{\times (d-1)}$. Any point $z\in\mathcal{Z}$ can be written as $z=(x, \delta_1,\cdots, \delta_{d-1})$ for $x\in \mathcal{X}_{\overline{V}}$ and $\delta_i:K^{\times}\rightarrow E(z)^{\times}.$ Let 
$$p:\mathcal{Z}\rightarrow \mathcal{X}_{\overline{V}}: (x,\delta_1,\cdots,\delta_{d-1})\mapsto x, $$ and, for $1\leqq i\leqq d-1$,  
$$q_i:\mathcal{Z}\rightarrow \mathcal{T}:(x,\delta_1,\cdots,\delta_{d-1})\mapsto \delta_i$$ be the projections. 
We set $N:=p^{*}(M)$, $P_{N,\sigma}(T):=p^*(P_{M,\sigma}(T))\in \Gamma(\mathcal{Z}, \mathcal{O}_{\mathcal{Z}})[T]$ and 
$$\delta_i^{\mathrm{univ}}:K^{\times} 
\xrightarrow{\delta_{\mathcal{T}}^{\mathrm{univ}}} \Gamma(\mathcal{T}, \mathcal{O}_{\mathcal{T}})^{\times}\xrightarrow{q_i^*}  \Gamma(\mathcal{Z}, \mathcal{O}_{\mathcal{Z}})^{\times}$$  and
$$Y_i:=q_i^*(\delta^{\mathrm{univ}}_{\mathcal{T}}(\pi_K))\in \Gamma(\mathcal{Z}, \mathcal{O}_{\mathcal{Z}})^{\times}.$$ 
For $i=d$, we define 
$$\delta_d^{\mathrm{univ}}:=(\mathrm{det}(N)\circ\mathrm{rec}_K)/\delta^{\mathrm{univ}}_{[d-1]}:K^{\times}\rightarrow \Gamma(\mathcal{Z},\mathcal{O}_{\mathcal{Z}})^{\times},$$ 
where we set $\delta^{\mathrm{univ}}_{[d-1]}:=\prod_{i=1}^{d-1}\delta^{\mathrm{univ}}_{i}$.
 For any $1\leqq i\leqq d-1$, we define a continuous character $$\widetilde{\delta}^{\mathrm{univ}}_i:G_K^{\mathrm{ab}}\rightarrow \Gamma(\mathcal{Z}, \mathcal{O}_{\mathcal{Z}})^{\times}$$ such that $$\widetilde{\delta}^{\mathrm{univ}}_i\circ \mathrm{rec}_K|_{\mathcal{O}_K^{\times}}=\delta^{\mathrm{univ}}_i|_{\mathcal{O}_K^{\times}}\text{  and }\widetilde{\delta}_i^{\mathrm{univ}}(\mathrm{rec}_K(\pi_K))=1,$$ 
 and define a unramified homomorphism
 $$\delta_{Y_i}:K^{\times}\rightarrow \Gamma(\mathcal{Z},\mathcal{O}_{\mathcal{Z}})^{\times}$$ 
 such that $\delta_{Y_i}|_{\mathcal{O}_K^{\times}}=1$ and $\delta_{Y_i}(\pi_K)=Y_i$.

Under these notations, we define a Zariski closed subspace 
$$\mathcal{Z}_0\subseteq \mathcal{Z}$$ the largest Zariski closed subspace such that the equality $$P_{N,\sigma}(T)=\prod_{i=1}^{d}(T-k(\delta^{\mathrm{univ}}_i)_{\sigma}) $$ holds on $\mathcal{Z}_0$ 
for any $\sigma\in \mathcal{P}$, i.e. if we denote $$P_{N,\sigma}(T)-
\prod_{i=1}^{d}(T-k(\delta^{\mathrm{univ}}_i)_{\sigma}) :=a_{d-1,\sigma}T^{d-1}+\cdots+a_{0,\sigma}\in 
\Gamma(\mathcal{Z},\mathcal{O}_{\mathcal{Z}})[T],$$ then $\mathcal{Z}_0$ is defined 
by the ideal generated by $\{a_{i,\sigma}\}_{0\leqq i\leqq d-1,\sigma\in \mathcal{P}}$. For any $1\leqq i\leqq d-1$, let $\wedge^iN$ be the $i$-th exterior product of $N$ over $\mathcal{O}_{\mathcal{Z}}$. For 
each $1\leqq i\leqq d-1$, set $\delta^{\mathrm{univ}}_{[i]}:=\prod_{j=1}^i\delta^{\mathrm{univ}}_j$, 
then the $\sigma$-part of Sen's polynomial of $$N_i:=(\wedge^i N
\otimes_{\mathcal{O}_{\mathcal{Z}}} \mathcal{O}_{\mathcal{Z}}(\widetilde{\delta^{\mathrm{univ}}_{[i]}}^{-1}))|_{\mathcal{Z}_0}$$ is written 
by $TQ_{i,\sigma}(T)$ for a monic polynomial $Q_{i,\sigma}(T)\in \mathcal{O}_{\mathcal{Z}_0}[T]$ because 
the $\sigma$-part of the Hodge-Tate weight of $\widetilde{\delta_i^{\mathrm{univ}}}$ is $k(\delta^{\mathrm{univ}}_i)_{\sigma}$ by Proposition \ref{a-3}. Hence, we can apply Theorem \ref{1.1} to this situation, more precisely, we obtain the 
following corollary. Set $Y_{[i]}:=\prod_{j=1}^iY_j$ for $1\leqq i\leqq d-1$.

\begin{corollary}\label{1.3}
Under the above situation, there exists a unique Zariski closed subspace $$\mathcal{E}_{\overline{V}}:=\mathcal{Z}_{0,fs}\subseteq \mathcal{Z}_0$$
 satisfying the following conditions (1) and (2).
\begin{itemize}
\item[(1)]For any $1\leqq i\leqq d-1$, $\sigma\in \mathcal{P}$ and $j\in \mathbb{Z}_{\leqq 0}$, $\mathcal{E}_{\overline{V},Q_{i,\sigma}(j)}$ is 
scheme theoretically dense in $\mathcal{E}_{\overline{V}}$.
\item[(2)] For any $E$-morphism $f:\mathrm{Spm}(R)\rightarrow \mathcal{Z}_0$ which is $Y_i$-small for all 
$1\leqq i\leqq d-1$ and factors through $\mathcal{Z}_{0,Q_{i,\sigma}(j)}$ for any 
$1\leqq i\leqq d-1$, $\sigma\in \mathcal{P}$ and $j\in \mathbb{Z}_{\leqq 0}$, the following conditions (i) and (ii) 
are equivalent.
\begin{itemize}
\item[(i)]$f$ factors through $f:\mathrm{Spm}(R)\rightarrow \mathcal{E}_{\overline{V}}\hookrightarrow \mathcal{Z}_0$.
\item[(ii)]For any $1\leqq i\leqq d-1$, any $R$-linear $G_K$-equivariant map 
$$h:f^*(N_i^{\vee})\rightarrow 
\bold{B}^+_{\mathrm{dR}}\hat{\otimes}_{\mathbb{Q}_p} R$$ factors through the natural inclusion 
$$h:f^*(N_i^{\vee})\rightarrow K\otimes_{K_0}(\bold{B}^+_{\mathrm{max}}\hat{\otimes}_{\mathbb{Q}_p}R)^{\varphi^f=Y_{[i]}}
\hookrightarrow \bold{B}^+_{\mathrm{dR}}\hat{\otimes}_{\mathbb{Q}_p}R.$$
\end{itemize}
\end{itemize}

\end{corollary}
\begin{rem}
By the definition, we can easily check that $f:\mathrm{Spm}(R)\rightarrow \mathcal{Z}_0$ is $Y_i$-small for any $1\leqq i\leqq d-1$ if and only if 
$f$ is $Y_{[i]}$-small for any $1\leqq i\leqq d-1$.
\end{rem}

Each point $x\in \mathcal{X}_{\overline{V}}$ corresponds to an $E(x)$-representation 
$V_x$ of $G_K$ such that there exists a $G_K$-stable $\mathcal{O}_{E(x)}$-lattice $T_x\subseteq V_x$ which satisfies 
$T_x/\pi_{E(x)}T_x\isom \overline{V}\otimes_{\mathbb{F}}(\mathcal{O}_{E(x)}/\pi_{E(x)}\mathcal{O}_{E(x)})$. We assume that, for 
a finite extension $E'$ of $E(x)$,  $V_x\otimes_{E(x)}E'$ is a split trianguline $E'$-representation with a triangulation 
$$T_x:0\subseteq W_1\subseteq W_2 \subseteq \cdots \subseteq W_d:=W(V_x\otimes_{E(x)}E').$$
 We denote by $\{\delta_i\}_{i=1}^{d}$ the parameter 
of $T_x$, i.e. $\delta_i:K^{\times}\rightarrow E^{' \times}$ satisfies $W_i/W_{i-1}\isom W(\delta_i)$
for any $i$. By Proposition \ref{a-3},  the couple $(V_x,T_x)$ 
defines an $E'$-rational point $$z:=z_{(V_x,T_x)}:=(x, \delta_1,\delta_2,\cdots,\delta_{d-1})\in\mathcal{Z}_0(E').$$ By Galois descent, all these are defined over $E(z)(\subseteq E')$, i.e. 
the $E'$-triangulation $T_x$ descends to an $E(z)(\subseteq E')$-triangulation $T_x$
of $V_x\otimes_{E(x)}E(z)$ with the same parameter. Hence, if we write 
$z:=z_{(V_x,T_x)}\in\mathcal{Z}_0$, then we always assume that $V_x$ is a split trianguline $E(z)$-representation with 
an $E(z)$-triangulation $T_x$.


\begin{prop}\label{1.4}
Let $(V_x,T_x)$ be a couple as above which satisfies the following conditions $($put $z:=z_{(V_x,T_x)}$$)$,
\begin{itemize}
\item[(0)]$\mathrm{End}_{E(z)[G_K]}(V_x)=E(z)$,
\item[(1)]$\delta_j/\delta_i\not=\prod_{\sigma\in \mathcal{P}}\sigma^{k_{\sigma}}$ for any $1\leqq i<j\leqq d$ and $\{k_{\sigma}\}_{\sigma\in \mathcal{P}}\in \prod_{\sigma\in \mathcal{P}}\mathbb{Z}_{\leqq 0}$,
\item[(2)]$\delta_i/\delta_j\not=|N_{K/\mathbb{Q}_p}|_p\prod_{\sigma\in \mathcal{P}}
\sigma^{k_{\sigma}}$ for any $1\leqq i<j\leqq d$ and $\{k_{\sigma}\}_{\sigma\in \mathcal{P}}\in \prod_{\sigma\in \mathcal{P}}\mathbb{Z}_{\geqq 1}$,

\end{itemize}
then the point $z\in\mathcal{Z}_0$ is contained in $\mathcal{E}_{\overline{V}}$.

\end{prop}
\begin{proof}
First, because the construction of $X_{fs}$ commutes with base changes $E\mapsto E'$ by Lemma 3.10 of \cite{Na10} (more precisely, we can prove this lemma for our modified $X_{fs}$ in the same way), we may assume 
that $E(z)=E$. We prove the proposition under this assumption. 
By the conditions (0), (1) and Proposition 2.34 of [Na10], $D_{V_x}$ and $D_{V_x,T_x}$ are representable
by $R_{V_x}$ and  $R_{V_x,T_x}$ respectively. We denote by $V_x^{\mathrm{univ}}$ 
the universal deformation 
of $V_x$ over $R_{V_x}$, and denote by 
$$T_x^{\mathrm{univ}}: 0\subseteq W^{\mathrm{univ}}_1\subseteq W_{2}^{\mathrm{univ}}\subseteq \cdots 
\subseteq W^{\mathrm{univ}}_d=W(V^{\mathrm{univ}}_x)\otimes_{R_{V_x}}R_{V_x,T_x}$$ 
the universal triangulation,  and denote by $\{\delta_{i,x}\}_{i=1}^d$
 the parameter of $T_x^{\mathrm{univ}}$.

Because we have canonical isomorphisms
$$\hat{\mathcal{O}}_{\mathcal{X}_{\overline{V}},x}\isom R_{V_x}\text{ and }M\otimes_{\mathcal{O}_{\mathcal{X}_{\overline{V}}}}\hat{\mathcal{O}}_{\mathcal{X}_{\overline{V}},x}\isom V^{\mathrm{univ}}_x$$ by Proposition 9.5 of \cite{Ki03}, we can define a morphism of $E$-algebras 
 $$g:\hat{\mathcal{O}}_{\mathcal{Z},z}
\isom \hat{\mathcal{O}}_{\mathcal{X}_{\overline{V}},x}\hat{\otimes}_E (\hat{\otimes}_{i=1}^{d-1}\hat{\mathcal{O}}_{\mathcal{T},\delta_i})
\isom R_{V_x}\hat{\otimes}_{E}(\hat{\otimes}_{i=1}^{d-1}\hat{\mathcal{O}}_{\mathcal{T},\delta_i})\twoheadrightarrow 
R_{V_x,T_x}\hat{\otimes}_E(\hat{\otimes}_{i=1}^{d-1}\hat{\mathcal{O}}_{\mathcal{T},\delta_i}),$$ 
where the first isomorphism is the natural one and the second isomorphism is induced by 
the above isomorphism $\hat{\mathcal{O}}_{\mathcal{X}_{\overline{V}},x}\isom R_{V_x}$ and the third surjection is induced by the natural quotient map $R_{V_x}\rightarrow R_{V_x,T_x}$.

We define an ideal $I$ of 
$R_{V_x,T_x}\hat{\otimes}_E (\hat{\otimes}_{i=1}^{d-1} \hat{\mathcal{O}}_{\mathcal{T}, \delta_i})$ which is generated by 
$$\{\delta_{i,x}(a)\otimes 1-g(\delta_i^{\mathrm{univ}}(a))| 1\leqq i\leqq d-1, a\in K^{\times} \}.$$ We set
$$R_z:= (R_{V_x,T_x}\hat{\otimes}_E (\hat{\otimes}_{i=1}^{d-1} \hat{\mathcal{O}}_{\mathcal{T}, \delta_i}))/I.$$
Then, the natural map $$R_{V_x,T_x}\rightarrow R_{V_x,T_x}\hat{\otimes}_E (\hat{\otimes}_{i=1}^d \hat{\mathcal{O}}_{\mathcal{T},\delta_i}) \twoheadrightarrow R_z: y\mapsto \overline{y\otimes 1}$$ 
is an isomorphism, and the universal parameter $\{\delta_{i,x}\}_{i=1}^{d}$ is equal to 
$\{\delta^{\mathrm{univ}}_i\}_{i=1}^{d}$ on $R_z$.
We define a morphism 
$$h:\hat{\mathcal{O}}_{\mathcal{Z},z}\xrightarrow{g} R_{V_x,T_x}\hat{\otimes}_E
(\hat{\otimes}_{i=1}^{d-1}\hat{\mathcal{O}}_{\mathcal{T},\delta_i})\twoheadrightarrow R_z.$$
 Because we have an isomorphism $M\otimes_{\mathcal{O}_{\mathcal{X}_{\overline{V}}}}\hat{\mathcal{O}}_{\mathcal{X}_{\overline{V}},x}\isom V^{\mathrm{univ}}_x$, $N\otimes_{\mathcal{O}_{\mathcal{Z}}}R_z$ is isomorphic to 
the universal trianguline deformation of $V_x$ over $R_{V_x,T_x}\isom R_z$. 
Hence, the $\sigma$-part of Sen's polynomial of $N\otimes_{\mathcal{O}_{\mathcal{Z}}}R_z$ is equal to $\prod_{i=1}^d(T-k(\delta_i^{\mathrm{univ}})_{\sigma})$, so the natural morphism $$\mathrm{Spm}(R_{z}/\mathfrak{m}^n)
\rightarrow \mathcal{Z}$$ factors through $$f_n:\mathrm{Spm}(R_z/\mathfrak{m}^n)\rightarrow \mathcal{Z}_0$$ for any $n\geqq 1$, where $\mathfrak{m}\subseteq R_z$ is the maximal ideal. 
We claim that $f_n$ also factors thorough $\mathcal{E}_{\overline{V}}$ for any $n$, which proves the proposition because the point of $\mathcal{Z}_0$
determined by $f_1$ is equal to $z$.

To show this claim, we note that $R_z$ is 
formally smooth over $E$ by the condition (2) and by Proposition 2.36 of \cite{Na10}. 
In particular, $R_z$ is a domain.
Then, by the proof of Theorem \ref{1.1}, it suffices to show the following lemma.

\end{proof}

\begin{lemma}\label{1.5}
The following hold.
\begin{itemize}
\item[(1)]$Q_{i,\sigma}(j)$ is non-zero in $R_z$ for any $1\leqq i\leqq d-1$, 
$\sigma\in \mathcal{P}$ and $j\in \mathbb{Z}_{\leqq 0}$, 
\item[(2)]For any $1\leqq i\leqq d-1$, $\sigma\in\mathcal{P}$ and $k\in \mathbb{Z}_{\geqq 1}$, the natural map 
$$\varprojlim_{n\geqq 1}(\bold{B}^+_{\mathrm{max},K}\otimes_{K,\sigma}f_n^*(N_i))^{G_K,\varphi_K=Y_{[i]}}\rightarrow 
\varprojlim_{n\geqq 1}(\bold{B}^+_{\mathrm{dR}}/t^k\bold{B}^+_{\mathrm{dR}}\otimes_{K,\sigma}f_n^*(N_i))^{G_K}$$ 
is a surjection.
\end{itemize}
\end{lemma}
\begin{proof}
If we can prove (i), then (ii) can be proved in the same way as in the proof of Proposition 2.8 of \cite{Ki03}. 
We prove (i). Because the $\sigma$-part of the generalized Hodge-Tate weights of $f_n^*(N)$ is (modulo $\mathfrak{m}^n$ of) $\{k(\delta_{i,x})_{\sigma}\}_{i=1}^d$, 
the $\sigma$-part of the generalized Hodge-Tate weights of $f_n^*(N_i)$ is 
$$\{k(\delta_{J,x})_{\sigma}-k(\delta_{[i],x})_{\sigma} |J\subseteq [d] \text{ such that }\sharp(J)=i\},$$
 for each $1\leqq i\leqq d-1$, where we define 
 $\delta_{J,x}:=\prod_{l=1}^i\delta_{j_l,x}$ for each subset $J=\{j_1,j_2,\cdots,j_i\}$ of $[d]$. Therefore, 
the $\sigma$-part of Sen's polynomial of $f_n^*(N_i)$ is equal to 
$$\prod_{J\subseteq [d],\sharp(J)=i}(T-k(\delta_{J,x})_{\sigma}+k(\delta_{[i],x})_{\sigma}), $$ so the 
polynomial $Q_{i,\sigma}(T)$ for $f_n^*(N)$ is equal to
$$\prod_{J\not=[i]\subseteq [d], \sharp(J)=i}
(T-k(\delta_{J,x})_{\sigma}+k(\delta_{[i],x})_{\sigma}).$$
Hence, it suffices to show the following lemma.

\end{proof}
\begin{lemma}\label{1.6}
For any $1\leqq i\leqq d-1$ and for any subset $J\not=[i]$ such that $\sharp(J)=i$, then 
$k(\delta_{J,x})_{\sigma}-k(\delta_{[i],x})_{\sigma}$ is not constant, i.e. not contained in $E$ for any 
$\sigma\in\mathcal{P}$.
\end{lemma}
\begin{proof}
For any continuous homomorphism $\delta:\mathcal{O}_K^{\times}\rightarrow E^{\times}$, we define a functor 
$D_{\delta}:\mathcal{C}_E\rightarrow (Sets)$ by
\begin{multline*}
D_{\delta}(A):=\{\delta_A:\mathcal{O}_K^{\times}\rightarrow A^{\times}|\text{continuous homomorphisms }\\
\text{such that } 
\delta_A (\text{ mod }\mathfrak{m}_A)\equiv\delta \},
\end{multline*}
 for $A\in \mathcal{C}_E$.
Then, $D_{\delta}$ is representable by  a ring $R_{\delta}$ which is formally smooth over $E$ of its dimension equal to $[K:\mathbb{Q}_p]$. If we denote by $\delta^{\mathrm{univ}}:\mathcal{O}_K^{\times}\rightarrow R_{\delta}^{\times}$ the 
universal deformation of $\delta$, then the $\sigma$-part of the generalized 
Hodge-Tate weight $k(\delta^{\mathrm{univ}})_{\sigma}$ is not constant for any $\sigma\in\mathcal{P}$ by Lemma 3.19 of \cite{Na10}. 

For each subset $J\not=[i]$ of $[d]$ such that $\sharp(J)=i$, we set $\delta_{J,0}:=\delta_{J}\cdot\delta_{[i]}^{-1}|_{\mathcal{O}_K^{\times}}$ and 
define a morphism of functors $f_{J}:D_{V_x,T_x}\rightarrow D_{\delta_{J,0}}$ 
by $$f_{J}([(V_A,T_A)]):=(\delta_{J,A}\cdot\delta_{[i],A}^{-1})|_{\mathcal{O}_K^{\times}},$$ where, for  the parameter $\{\delta_{j,A}\}_{j=1}^d$ of the $A$-triangulation 
$T_A$, we set $\delta_{J,A}:=\prod_{l=1}^{i}\delta_{j_l,A}$ for  each subset $J=\{j_1,\cdots,j_i\}$ of $[d]$. We claim that this morphism is formally smooth, which proves the lemma 
because then $k(\delta_{J,x})_{\sigma}-k(\delta_{[i],x})_{\sigma}$ is equal to 
$f_{J}^*(k(\delta^{\mathrm{univ}}_{J,0})_{\sigma})$ where $f_{J}^*: R_{\delta_{J,0}}\allowbreak \hookrightarrow R_{V_x,T_x}$ is 
the map induced by $f_{J}$ which is an injection by the formally smoothness of $f_{J}$. 
We prove the claim. Let $A$ be an object of $\mathcal{C}_E$ and $I\subseteq A$ be an ideal of $A$ such that $I^2=0$. For $[(V_{A/I},T_{A/I})]\in D_{V_x,T_x}(A/I)$ and $\delta_A\in D_{\delta_{J,0}}(A)$ such that $\delta_{J,A/I}\cdot\delta_{[i],A/I}^{-1}|_{\mathcal{O}_K^{\times}}\equiv\delta_A (\text{mod } I)$, where the parameter of the $A/I$-triangulation $T_{A/I}$ is $\{\delta_{j,A/I}\}_{j=1}^d$, then it suffices to show that there exists $[(V_A,T_A)]\in D_{V_x,T_x}(A)$ a lift of $[(V_{A/I},T_{A/I})]$ such that 
$\delta_{J,A}\cdot\delta_{[i],A}^{-1}|_{\mathcal{O}_K^{\times}}=\delta_A$.

Because $D_{\delta}$ is formally smooth for any $\delta$, we can take a lift 
$\{\delta_{j,A}:K^{\times}\rightarrow A^{\times}\}_{j=1}^d$ of $\{\delta_{j,A/I}\}_{j=1}^d$ such that 
$\delta_{J,A}\cdot\delta_{[i],A}|_{\mathcal{O}_K^{\times}}=\delta_A$. Because we have 
$\mathrm{H}^2(G_K, W(\delta_i/\delta_j))=0$ for any $i<j$ by the condition (ii) and
Proposition 2.9 of \cite{Na10}, the natural map 
$$\mathrm{H}^1(G_K, W(\delta_{1,A}/\delta_{2,A}))\rightarrow 
\mathrm{H}^1(G_K, W(\delta_{1,A/I}/\delta_{2,A/I}))$$ is a surjection. Hence the sub $(A/I)$-$B$-pair $W_{2,A/I}$ of $W(V_{A/I})$ lifts to 
an $A$-$B$-pair $W_{2,A}$ which is an extension of $W(\delta_{1,A})$ by $W(\delta_{2,A})$. 
Repeating this procedure, we can take a trianguline $A$-$B$-pair $(W_A,T_A)$ which is a
 lift of $(W(V_{A/I}),T_{A/I})$ and whose parameter is $\{\delta_{i,A}\}_{i=1}^d$. Moreover, there exists 
 an $A$-representation $V_A$ such that $W_A\isom W(V_A)$ by Proposition 1.5.6 of \cite{Ke08}, 
 which proves the formally smoothness of $f_{J}$.

\end{proof}
Next, we prove that the local structure of $\mathcal{E}_{\overline{V}}$ at $z_{(V_x,T_x)}$ can be 
described in terms of the trianguline deformation $D_{V_x,T_x}$.

\begin{thm}\label{1.7}
Let $z:=z_{(V_x,T_x)}$ be a point of $\mathcal{E}_{\overline{V}}$ satisfying all the conditions  in Proposition \ref{1.4}. Moreover, if $(V_x,T_x)$ satisfies one of the following additional conditions (1) or (2), 

\begin{itemize}
\item[(1)]$V_x$ is potentially crystalline and $$\{a\in \bold{D}_{\mathrm{cris}}((\wedge^i V_x)(\widetilde{\delta}_{[i]}^{-1}))| 
\exists n\geqq 1\text{ such that }  (\varphi^f-\delta_{[i]}(\pi_K))^na=0\}$$ is a free of rank one $K_0\otimes_{\mathbb{Q}_p}E$-module 
for any $1\leqq i\leqq d-1$,
\item[(2)]For any $1\leqq i\leqq d-1$, and for any subset $J\not=[i]$ of $[d]$ such that $\sharp(J)=i$, 
we have  $k(\delta_{J})_{\sigma} -k(\delta_{[i]})_{\sigma}\not\in \mathbb{Z}_{\leqq 0}$ for any $\sigma\in \mathcal{P}$,

\end{itemize}

then there exists a canonical isomorphism $\hat{\mathcal{O}}_{\mathcal{E}_{\overline{V}}, z}\isom 
R_{V_x,T_x}$. In particular, $\mathcal{E}_{\overline{V}}$ is smooth of its dimension $[K:\mathbb{Q}_p]\frac{d(d+1)}{2} +1$ at $z$.

\end{thm}

\begin{proof}
 First, we claim that $D_0:=\bold{D}^+_{\mathrm{cris}}((\wedge^iV_x)(\widetilde{\delta}_{[i]}^{-1}))
^{\varphi^f=\delta_{[i]}(\pi_K)}$ is a sub $E$-filtered $\varphi$-module of $\bold{D}_{\mathrm{cris}}((\wedge^iV_x)(\widetilde{\delta}_{[i]}^{-1}))$ of rank one such that $\mathrm{Fil}^1K\otimes_{K_0}D_0=0.$
Because we have a natural inclusion 
$\bold{D}_{\mathrm{cris}}(W(\delta_{(\delta_{[i]}(\pi_K))}))\subseteq D_0$ which is induced  from  $\wedge^i(T_x)$ 
by Lemma 3.8 of \cite{Na10}, it suffices to show that $D_0$ is at most rank one. This is trivial for the condition (1). For (2), we assume that 
$D_0$ is not of rank one. Then, if we denote by $W'$
the cokernel of the natural injection $W(\delta_{(\delta_{[i]}(\pi_K))})\hookrightarrow 
W((\wedge^iV_x)(\widetilde{\delta}_{[i]}^{-1}))$ which is induced from $\wedge^i(T_x)$, 
the image of $D_0$ in $\bold{D}_{\mathrm{cris}}(W')$ is non zero. In particular, there exists a rank one 
$E$-filtered $\varphi$-submodule $D$ of $\bold{D}_{\mathrm{cris}}(W')^{\varphi^f=\delta_{[i
]}(\pi_K)}$ such that $\mathrm{Fil}^0K\otimes_{K_0}D=
K\otimes_{K_0}D$. This implies that $W'$ has a Hodge-Tate weight $\{k_{\sigma}\}_{\sigma\in \mathcal{P}}$ ($k_{\sigma}\leqq 0$). However, for any $\sigma$, the $\sigma$-part of 
the generalized Hodge-Tate weights of $W'$ are 
$\{k(\delta_{J})_{\sigma}-k(\delta_{[i]})_{\sigma}\}_{J}$ where $J$ runs through the subsets 
of $[d]$ such that $J\not=[i]$ and $\sharp(J)=i$, any of which is not negative integer 
by the assumption (2), which is a contradiction. We finish the proof of the claim.

We begin the proof of the theorem. As in the proof of Proposition \ref{1.4}, we may assume that $E(z)=E$.
By the proof of Proposition \ref{1.4}, we have already showed that there exists a natural local morphism 
$$\hat{\mathcal{O}}_{\mathcal{E}_{\overline{V}},z}\rightarrow R_z\isom R_{V_x,T_x}.$$ 
We construct the inverse as follows.
Let $\mathrm{Spm}(R)\subseteq \mathcal{E}_{\overline{V}}$ be an affinoid neighborhood of $z$ which is 
$Y_i$-small for any $1\leqq i\leqq d-1$, which  is possible because $z$ is an $E$-rational point so we have 
$Y_i(z)\in E$ and 
we can take $\mathrm{Spm}(R)$ such that $|\frac{Y_i}{Y_i(z)}-1|<1$ on $\mathrm{Spm}(R)$. We take a sufficiently large $k\in\mathbb{Z}_{\geqq 1}$ such that, for any 
$\sigma\in \mathcal{P}$ and $1\leqq i\leqq d-1$,
there exists a short exact sequence of Banach $R$-modules with the property (Pr)
$$0\rightarrow (\bold{B}^+_{\mathrm{max},K}\hat{\otimes}_{K,\sigma}R)^{\varphi_K=Y_{[i]}}\rightarrow 
\bold{B}^+_{\mathrm{dR}}/t^k\bold{B}^+_{\mathrm{dR}}\hat{\otimes}_{K,\sigma}R\rightarrow U_{i,\sigma}\rightarrow 0$$ 
for a Banach $R$-module $U_{i,\sigma}$ with the property (Pr) (see Proposition 3.7 of \cite{Na10}). By Proposition \ref{1.2},  we have an isomorphism
$$(\bold{B}^+_{\mathrm{max},K}\hat{\otimes}_{K,\sigma}(N_i\otimes_{\mathcal{O}_{\mathcal{Z}_0}}R))^{G_K, \varphi_K=Y_{[i]}}\isom 
(\bold{B}^+_{\mathrm{dR}}/t^k\bold{B}^+_{\mathrm{dR}}\hat{\otimes}_{K,\sigma}(N_i\otimes_{\mathcal{O}_{\mathcal{Z}_0}}R))^{G_K}$$ of locally free 
$R$-modules of rank one
for any $\sigma\in\mathcal{P}$ and $1\leqq i\leqq d-1$, and if we put $H_{i,\sigma}$ ($1\leqq i\leqq d-1$, $\sigma\in \mathcal{P}$) the smallest ideal of $R$ such that 
any $G_K$-equivariant $R$-linear map $h:N_i^{\vee}\rightarrow \bold{B}^+_{\mathrm{dR}}/t^k\bold{B}^+_{\mathrm{dR}}\hat{\otimes}_{K,\sigma} R$ 
factors through $\bold{B}^+_{\mathrm{dR}}/t^k\bold{B}^+_{\mathrm{dR}}\hat{\otimes}_{K,\sigma}H_{i,\sigma}$ and  put $H:=\prod_{1\leqq i\leqq d-1,\sigma\in\mathcal{P}}H_{i,\sigma}$, then 
$\mathrm{Spm}(R)\setminus V(H)$ and $\mathrm{Spm}(R)\setminus V(H_{i,\sigma})$ are scheme theoretically dense in $\mathrm{Spm}(R)$.
Under this situation, we denote by $\widetilde{T}$ the blow up of $\mathrm{Spm}(R)$ along the ideal $H$, and 
denote by $f:\widetilde{T}\rightarrow \mathrm{Spm}(R)$ 
the canonical projection. 
We claim that, for any $\widetilde{z}\in \widetilde{T}$ such that $f(\widetilde{z})=z$, 
$N\otimes_{\mathcal{O}_{\mathcal{Z}}}\mathcal{O}_{\widetilde{T},\widetilde{z}}/\mathfrak{m}_{\widetilde{z}}^n$ is a trianguline deformation of $(V_x\otimes_E E(\widetilde{z}),T_x\otimes_E E(\widetilde{z}))$ 
over $\mathcal{O}_{\widetilde{T},\widetilde{z}}/\mathfrak{m}^n_{\widetilde{z}}$ for any $n\in\mathbb{Z}_{\geqq 1}$.
To prove this claim, by the previous claim, we first note that $$\bold{D}^+_{\mathrm{cris}}(N_i\otimes_{\mathcal{O}_{\mathcal{Z}_0}}E(\widetilde{z}))^{\varphi^f=Y_{[i]}(\widetilde{z})}=\bold{D}^+_{\mathrm{cris}}((\wedge^iV_x)(\widetilde{\delta}_{[i]}^{-1}))
^{\varphi^f=\delta_{[i]}(\pi_K)}\otimes_E E(\widetilde{z})$$ is a free $K_0\otimes_{\mathbb{Q}_p}E(\widetilde{z})$-module of rank one and 
$$\mathrm{Fil}^1\bold{D}_{\mathrm{dR}}(N_i\otimes_{\mathcal{O}_{\mathcal{Z}_0}}E(\widetilde{z}))\cap K\otimes_{K_0}\bold{D}^+_{\mathrm{cris}}(N_i\otimes_{\mathcal{O}_{\mathcal{Z}_0}}E(\widetilde{z}))^{\varphi^f
=Y_{[i]}(\widetilde{z})}=0.$$
By the definition of blow up, there exists 
a non zero divisor $h_{i,\sigma}\in \hat{\mathcal{O}}_{\widetilde{T},\widetilde{z}}$ such that $H_{i,\sigma}\hat{\mathcal{O}}_{\widetilde{T},\widetilde{z}}=h_{i,\sigma}\hat{\mathcal{O}}_{\widetilde{T},\widetilde{z}}$ for any $1\leqq i\leqq d-1$ and $\sigma\in \mathcal{P}$. By the definition of 
$H_{i,\sigma}$, for any $1\leqq i\leqq d-1$ and $\sigma\in \mathcal{P}$, there exists a $G_K$-equivariant $R$-linear map 
$N_i^{\vee}\rightarrow (\bold{B}^+_{\mathrm{max},K}\hat{\otimes}_{K,\sigma}H_{i,\sigma})^{\varphi_K=Y_{[i]}}$ such that the composite with the map

\begin{multline*}
(\bold{B}^+_{\mathrm{max},K}\hat{\otimes}_{K,\sigma}H_{i,\sigma})^{\varphi_K=Y_{[i]}}\rightarrow (\bold{B}^+_{\mathrm{max},K}\hat{\otimes}_{K,\sigma}h_{i,\sigma}\hat{\mathcal{O}}_{\widetilde{T},\widetilde{z}})^{\varphi_K=Y_{[i]}} \\
  \isom (\bold{B}^+_{\mathrm{max},K}\hat{\otimes}_{K,\sigma}\hat{\mathcal{O}}_{\widetilde{T},\widetilde{z}})^{\varphi_K=Y_{[i]}}
\twoheadrightarrow (\bold{B}^+_{\mathrm{max},K}\otimes_{K,\sigma}E(\widetilde{z}))^{\varphi_K=\delta_{[i]}(\pi_K)}
\end{multline*}
 is non zero, where the isomorphism $$(\bold{B}^+_{\mathrm{max},K}\hat{\otimes}_{K,\sigma}h_{i,\sigma}\hat{\mathcal{O}}_{\widetilde{T},\widetilde{z}})^{\varphi_K=Y_{[i]}}\isom (\bold{B}^+_{\mathrm{max},K}\hat{\otimes}_{K,\sigma}\hat{\mathcal{O}}_{\widetilde{T},\widetilde{z}})^{\varphi_K=Y_{[i]}}$$ is given by $a\mapsto \frac{a}{h_{i,\sigma}}$. From these facts, 
 we can show by induction on 
 $n$ that
 $\bold{D}^+_{\mathrm{cris}}(N_i\otimes_{\mathcal{O}_{\mathcal{Z}_0}}\hat{\mathcal{O}}_{\widetilde{T},\widetilde{z}}/\mathfrak{m}_{\widetilde{z}}^n)^{\varphi_K=Y_{[i]}}$ is a free $K_0\otimes_{\mathbb{Q}_p}\hat{\mathcal{O}}_{\widetilde{T},\widetilde{z}}/\mathfrak{m}_{\widetilde{z}}^n$-module of rank one for any $1\leqq i\leqq d-1$ and $$K\otimes_{K_0}\bold{D}^+_{\mathrm{cris}}(N_i\otimes_{\mathcal{O}_{\mathcal{Z}_0}}\hat{\mathcal{O}}_{\widetilde{T},\widetilde{z}}/\mathfrak{m}_{\widetilde{z}}^n)^{\varphi^f=
 Y_{[i]}}\cap \mathrm{Fil}^1\bold{D}_{\mathrm{dR}}(N_i\otimes_{\mathcal{O}_{\mathcal{Z}_0}}\hat{\mathcal{O}}_{\widetilde{T},\widetilde{z}}/\mathfrak{m}_{\widetilde{z}}^n)=0$$
 for any $n\geqq 1$.

 By Proposition \ref{1.8} below, then $N\otimes_{\mathcal{O}_{Z}}\hat{\mathcal{O}}_{\widetilde{T},\widetilde{z}}/\mathfrak{m}_{\widetilde{z}}^n$ 
 is a trianguline deformation of $(V_x\otimes_E E(\widetilde{z}),T_x\otimes_E E(\widetilde{z}))$ over $\hat{\mathcal{O}}_{\widetilde{T},\widetilde{z}}/\mathfrak{m}_{\widetilde{z}}^n$ 
 whose parameter is $\{\delta_i^{\mathrm{univ}} (\text{mod }\mathfrak{m}_{\widetilde{z}}^n)\}_{i=1}^d$, so the natural map 
 $R_{V_x}\rightarrow \hat{\mathcal{O}}_{\widetilde{T},\widetilde{z}}$ factors through $R_{V_x}\rightarrow R_{V_x,T_x}\rightarrow \hat{\mathcal{O}}_{\widetilde{T},\widetilde{z}}$. 
 This shows that the natural map $R_{V_x}\rightarrow \hat{\mathcal{O}}_{\mathcal{E}_{\overline{V}},z}$ sends the kernel of the quotient map $R_{V_x}\rightarrow R_{V_x,T_x}$ to the kernel of the natural map 
 $$g:\hat{\mathcal{O}}_{\mathcal{E}_{\overline{V}},z}\rightarrow \prod_{\widetilde{z}\in \widetilde{T}, f(\widetilde{z})=z}\hat{\mathcal{O}}_{\widetilde{T},\widetilde{z}}.$$ 
 Because $g$ is an injection by Lemma 10.7 of \cite{Ki03} and by Proposition \ref{1.2} (2), the natural map $R_{V_x}\rightarrow \hat{\mathcal{O}}_{\mathcal{E}_{\overline{V}},z}$ also 
 factors thorough $R_{V_x}\rightarrow R_{V_x,T_x}\rightarrow \hat{\mathcal{O}}_{\mathcal{E}_{\overline{V}},z}$. We can easily check that this gives the desired inverse map. The last statement of the theorem follows from Proposition 2.39 of \cite{Na10}.

\end{proof}
The following proposition which we used in the above proof is a generalization of Theorem 2.5.6 of \cite{BeCh09} for any $K$.

\begin{prop}\label{1.8}
Let $V$ be a split trianguline $E$-representation of rank $d$ with a triangulation $T:0\subseteq 
W_1\subseteq W_2\subseteq\cdots\subseteq W_{d}:=W(V)$ 
whose parameter is $\{\delta_i\}_{i=1}^d$. 
We assume that $(V,T)$ satisfies one of conditions (1) or (2) of Theorem \ref{1.7}.
Let $A\in \mathcal{C}_E$, and let $V_A$ be a deformation of $V$ over $A$, and, for any $1\leqq i\leqq d$, let $\delta_{i,A}:K^{\times}\rightarrow A^{\times}$ be a continuous homomorphism which is a lift of $\delta_i$ 
 satisfying the following conditions,
 \begin{itemize}
 \item[(1)]for any $1\leqq i\leqq d-1$, $\bold{D}^+_{\mathrm{cris}}((\wedge^iV_A)(\widetilde{\delta}_{[i],A}^{-1}))^{\varphi^f=
 \delta_{[i],A}(\pi_K)}$ is a free $K_0\otimes_{\mathbb{Q}_p}A$-module of rank one,
 \item[(2)]for any $1\leqq i\leqq d-1$, the natural base change map $$\bold{D}^+_{\mathrm{cris}}((\wedge^iV_A)(\widetilde{\delta}_{[i],A}^{-1}))^{\varphi^f=
 \delta_{[i],A}(\pi_K)}\otimes_{A}E\rightarrow\bold{D}^+_{\mathrm{cris}}((\wedge^iV)(\widetilde{\delta}_{[i]}))^{\varphi^f=\delta_{[i]}(\pi_K)}$$ is isomorphism.
 \end{itemize}
 Then, $V_A$ has an $A$-triangulation $T_A$ such that $(V_A,T_A)$ is a deformation of 
 $(V,T)$ whose parameter is equal to $\{\delta_{i,A}\}_{i=1}^d$.

\end{prop}

\begin{proof}
The proof is of course essentially the same as that of \cite{BeCh09}, but we give the proof here for  convenience of readers.
%
By the claim in the proof of the above Theorem, for any $1\leqq i\leqq d-1$, we have
$$K\otimes_{K_0}\bold{D}^+_{\mathrm{cris}}((\wedge^iV)(\widetilde{\delta}_{[i]}^{-1}))^{\varphi^f=
\delta_{[i]}(\pi_K)}\cap \mathrm{Fil}^1\bold{D}_{\mathrm{dR}}((\wedge^iV)(\widetilde{\delta}_{[i]}^{-1}))=0.$$ By induction on the length of $A$, 
we can also show that 
$$K\otimes_{K_0}\bold{D}^+_{\mathrm{cris}}((\wedge^iV_A)(
\widetilde{\delta}_{[i],A}^{-1}))^{\varphi^f=\delta_{[i],A}(\pi_K)}
\cap \mathrm{Fil}^1\bold{D}_{\mathrm{dR}}((\wedge^iV_A)(
\widetilde{\delta}_{[i],A}^{-1}))=0.$$
From this and the condition (1), $\bold{D}^+_{\mathrm{cris}}((\wedge^iV_A)(\widetilde{\delta}_{[i],A}^{-1}))^{\varphi^f=\delta_{[i],A}(\pi_K)}$ is an 
$A$-filtered $\varphi$-module. Hence, by the condition (2) and Lemma 2.22 of \cite{Na10}, there exists an $A$-saturated inclusion 
$$h_i:W(\delta_{(\delta_{[i],A}(\pi_K))})
\hookrightarrow W((\wedge^iV_A)(\widetilde{\delta}_{[i],A}^{-1}))$$ such that 
$\bold{D}_{\mathrm{cris}}(h_i)$ corresponds to the canonical injection $$\bold{D}^+_{\mathrm{cris}}((\wedge^iV_A)(\widetilde{\delta}_{[i],A}^{-1}))
^{\varphi^f=\delta_{[i],A}}\hookrightarrow \bold{D}_{\mathrm{cris}}((\wedge^iV_A)(\widetilde{\delta}_{[i],A}^{-1})).$$
Twisting this injection $h_i$ by $\widetilde{\delta}_{[i],A}$, we obtain an $A$-saturated injection 
$$h'_i:W(\delta_{[i],A})\hookrightarrow W(\wedge^i V_A)$$ such that 
$h'_i (\text{mod }\mathfrak{m}_A)$ is equal to the canonical injection 
$$W(\delta_{[i]})\hookrightarrow W(\wedge^iV)$$ which is naturally induced by 
$\wedge^iT$. Using these facts, we show by induction on $i$ that there exists an $A$-saturated 
sub $A$-$B$-pair $W_{i,A}$ of $W(V_A)$ such that $W_{i-1,A}$ is a sub $A$-saturated 
$A$-$B$-pair of $W_{i,A}$ and $W_{i,A}/W_{i-1,A}\isom W(\delta_{i,A})$ and the image of  the inclusion 
$W_{i,A}\otimes_A E\hookrightarrow W(V_A)\otimes_A E=W(V)$ is equal to
$W_i$.
For $i=1$, we take $W_{1,A}$ as the image of the inclusion $h'_1:W(\delta_{1,A})\hookrightarrow W(V_A)$. 
We assume that we can take a filtration $W_{1,A}\subseteq W_{2,A}\subseteq \cdots \subseteq W_{i-1,A}\subseteq W(V_A)$ 
satisfying the above conditions. 
Denote by $W'_{i,A}$(resp. $W'_i$) the cokernel of the inclusion $W_{i-1,A}\hookrightarrow W(V_A)$ 
(resp. $W_{i-1}\hookrightarrow W(V)$).  Taking the $i$-th exterior product of $W(V_A)$ (resp. $W(V)$), 
we obtain a following short exact sequence of $A$-$B$-pairs

$$0\rightarrow (\wedge^{i-1}W_{i-1,A})\otimes W'_{i,A}\rightarrow W(\wedge^iV_A) \rightarrow W''_{i,A}\rightarrow 0$$
for an $A$-$B$-pair $W''_{i,A}$ (similarly for $W(\wedge^i V)$ for an $E$-$B$-pair $W''_i$). Under this situation, we claim that the 
map $W(\delta_{[i],A})\rightarrow W''_{i,A}$ which is defined as the composite of $h'_i$ with the canonical 
projection $W(\wedge^iV_A) \rightarrow W''_{i,A}$, is zero. By d\'evissage, it suffices to show that 
the natural map $W(\delta_{[i]})\rightarrow W''_{i}$ is zero.
We first prove this claim under the condition
$(1)$ of Theorem \ref{1.7}. Because we have $\bold{D}_{\mathrm{cris}}(W(\delta_{(\delta_{[i]}(\pi_K))}))^{\varphi^f=\delta_{[i]}(\pi_K)}\not=0$,  it suffices to show that $\bold{D}_{\mathrm{cris}}(W''_{i}(\widetilde{\delta}_{[i]}^{-1}))^{\varphi^f=
\delta_{[i]}(\pi_K)}=0$, which follows from the condition (1) of Theorem \ref{1.7}. 
We next prove the claim under the condition (2) of Theorem 
\ref{1.7}. If the map $W(\delta_{[i]})\rightarrow W''_{i}$ is non zero, we also have an injection 
$W(\delta_{(\delta_{[i]}(\pi_K))})\hookrightarrow W''_i(\widetilde{\delta}_{[i]}^{-1})$, but this injection implies that 
$W''_i(\widetilde{\delta}_{[i]}^{-1})$ has Hodge-Tate weights $\{k_{\sigma}\}_{\sigma\in \mathcal{P}}$ for some $k_{\sigma}\in \mathbb{Z}_{\leqq 0}$, which contradicts the condition (2) of Theorem \ref{1.7}.

By this claim, the map $W(\delta_{[i],A})\hookrightarrow W(\wedge^iV_A)$ factors through 
an $A$-saturated injection $W(\delta_{[i],A})\hookrightarrow (\wedge^{i-1}W_{i-1,A})\otimes W'_{i,A}$. Because we have a natural isomorphism $\wedge^{i-1}W_{i-1,A}
\isom W(\delta_{[i-1],A})$, we obtain an $A$-saturated injection $W(\delta_{i,A})\hookrightarrow W'_{i,A}$. If we define an $A$-$B$-pair 
$W_{i,A}(\subseteq W(V_A))$ as the inverse image of $W(\delta_{i,A})(\subseteq W'_{i,A})$ by the natural projection $W(V_A)\rightarrow W'_{i,A}$, $W_{i,A}$ satisfies all the desired properties, which proves the proposition.

 \end{proof}
 
 \subsection{density of the crystalline points in the local eigenvariety}
 
 We next study the projection map 
 $\mathcal{E}_{\overline{V}}\rightarrow \mathcal{W}^{\times d}$ which we define below. 
 In the case of eigenvarieties, the projection to the weight space is very important for the application 
 to the classicality theorem of overconvergent modular forms, which says that any  overconvergent modular forms with a sufficiently large weight with respect to the slope of $U_p$-eigenvalue is classical modular form. This theorem ensures that the set of classical points is Zariski dense in the corresponding eigenvariety. In this subsection, we prove the local analogues of these properties for the local eigenvariety $\mathcal{E}_{\overline{V}}$.
 
 We first define a morphism $f_0:\mathcal{Z}\rightarrow \mathcal{W}^{\times d}$ by $f(z):=(\delta_1|_{\mathcal{O}_K^{\times}},\cdots,\delta_{d-1}|_{\mathcal{O}_K^{\times}},
 \delta_d|_{\mathcal{O}_K^{\times}})$ for $z:=(x,\delta_1,\cdots,\delta_{d-1})\in\mathcal{Z}$, where, 
 if $x\in X$ corresponds to an $E(x)$-representation $V_x$, we define 
 $\delta_d:=(\mathrm{det}(V_x)\circ\mathrm{rec}_K)\cdot\delta_{[d-1]}^{-1}$. 
 We also denote by $$f:\mathcal{E}_{\overline{V}}\rightarrow \mathcal{W}^{\times d}$$ the composition of $f_0$ with the canonical immersion $\mathcal{E}_{\overline{V}}\hookrightarrow Z$.
 
 \begin{prop}\label{1.9}
 Let $z:=z_{(V_x,T_x)}$ be a point of $\mathcal{E}_{\overline{V}}$ satisfying all the conditions in Theorem \ref{1.7}.
 Then, $f$ is smooth at $z$.
 
 \end{prop}
 \begin{proof}
 By Theorem \ref{1.7}, we have 
 an isomorphism $\widehat{\mathcal{O}}_{\mathcal{E}_{\overline{V}},z}\isom R_{V_x,T_x}$.
 Moreover, we have the following natural isomorphism
 $$\widehat{\mathcal{O}}_{\mathcal{W}^{\times d},f(z)}\isom 
 \widehat{\mathcal{O}}_{\mathcal{W},(\delta_1|_{\mathcal{O}_K^{\times}})}\hat{\otimes}_{E(z)}\cdots\hat{\otimes}_{E(z)}\widehat{\mathcal{O}}_{\mathcal{W},(\delta_d|_{\mathcal{O}_K^{\times}})}
 \isom R_{(\delta_1|_{\mathcal{O}_K^{\times}})}\hat{\otimes}_{E(z)}\cdots\hat{\otimes}_{E(z)} R_{(\delta_d|_{\mathcal{O}_K^{\times}})}$$ where $R_{(\delta_i|_{\mathcal{O}_K^{\times}})}$ is the universal deformation ring defined in 
 the proof of Lemma \ref{1.6}. If we identify 
 $$D_{V_x,T_x}(A)=\mathrm{Spf}(R_{V_x,T_x})(A)$$ and $$D_{(\delta_1|_{\mathcal{O}_K^{\times}})}(A)\times\cdots\times D_{(\delta_d|_{\mathcal{O}_K^{\times}})}(A)
 =\mathrm{Spf}(R_{(\delta_1|_{\mathcal{O}_K^{\times}})}\hat{\otimes}_{E(z)}\cdots\hat{\otimes}_{E(z)}R_{(\delta_d|_{\mathcal{O}_K^{\times}})})(A)$$ for $A\in\mathcal{C}_{E(z)}$,
 the local morphism at $z$ induced by $f$ is equal to the morphism $$f_{z}:\mathrm{Spf}(R_{V_x,T_x})\rightarrow \mathrm{Spf}(R_{(\delta_1|_{\mathcal{O}_K^{\times}})}\hat{\otimes}_{E(z)}\cdots\hat{\otimes}_{E(z)} R_{(\delta_d|_{\mathcal{O}_K^{\times}})})$$ whose $A$-valued points are  given by 
 $$D_{V_x,T_x}(A)\rightarrow D_{(\delta_1|_{\mathcal{O}_K^{\times}})}(A)\times\cdots\times D_{(\delta_d|_{\mathcal{O}_K^{\times}})}(A):[(V_A,T_A)]\mapsto 
 (\delta_{1,A}|_{\mathcal{O}_K^{\times}},\cdots,\delta_{d,A}|_{\mathcal{O}_K^{\times}}),$$ where $\{\delta_{i,A}\}_{i=1}^d$ is the parameter of $T_A$. We claim that this morphism is formally smooth. Let $A\in \mathcal{C}_{E(z)}$, and let $I\subseteq A$ be an ideal such that 
 $I\cdot\mathfrak{m}_A=0$. For any $[(V_{A/I},T_{A/I})]\in D_{V_x,T_x}(A/I)$ and 
 $(\delta'_{1,A},\delta'_{2,A},\cdots,\delta'_{d,A})\in D_{(\delta_1|_{\mathcal{O}_K^{\times}})}(A)\times\cdots\times D_{(\delta_d|_{\mathcal{O}_K^{\times}})}(A)$ such that $f_{z}([(V_{A/I},T_{A/I})])\equiv(\delta'_{1,A},\cdots,\delta'_{d,A})(\text{mod } I)$, it suffices to show that there exists a lift 
 $[(V_A,T_A)]\in D_{V_x,T_x}(A)$ of $[(V_{A/I},T_{A/I})]$ such that $f_z([(V_A,T_A)])=(\delta'_{1,A},\cdots,\delta'_{d,A})$.
 Take a lift $\{\delta_{i,A}\}_{i=1}^d$ ($\delta_{i,A}:K^{\times}\rightarrow A^{\times}$) of the parameter $\{\delta_{i,A/I}\}_{i=1}^d$ of $(V_{A/I},T_{A/I})$ such that $\delta_{i,A}|_{\mathcal{O}_K^{\times}}=\delta'_{i,A}$. Because we have $\mathrm{H}^2(G_K,W(\delta_i/\delta_j))
 =0$ for any $i<j$ by Proposition 2.9 of \cite{Na10}, we obtain an equality $\mathrm{H}^2(G_K, W_i(\delta_{i+1}^{-1}))=0$ for any $1\leqq i\leqq d-1$. 
 From this equality for $i=1$, we obtain a surjection $$\mathrm{H}^1(G_K,W(\delta_{1,A}/\delta_{2,A}))\rightarrow \mathrm{H}^1(G_K, W(\delta_{1,A/I}/\delta_{2,A/I})).$$ Hence, if we denote $T_{A/I}:0\subseteq W_{1.A/I}\subseteq W_{2,A/I}\subseteq 
 \cdots \subseteq W_{d,A/I}=W(V_{A/I})$, we can take a lift $[W_{2.A}]\in \mathrm{H}^1(G_K,W(\delta_{1,A}/\delta_{2,A}))$ of $[W_{2,A/I}]
 \in \mathrm{H}^1(G_K, W(\delta_{1,A}/\delta_{2,A}))$. Then, the natural map $\mathrm{H}^1(G_K, W_{2,A}(\delta_{3,A}^{-1}))\rightarrow 
 \mathrm{H}^1(G_K, W_{2,A/I}(\delta_{3,A/I}^{-1}))$ is also a surjection, hence we can take a lift $[W_{3,A}]\in \mathrm{H}^1(G_K, W_{2,A}(\delta_{3,A}^{-1}))$ of $[W_{3,A/I}]\in \mathrm{H}^1(G_K, W_{2,A/I}(\delta_{3,A/I}^{-1}))$.
 Repeating this procedure inductively, we obtain a trianguline $A$-$B$-pair $W_A$ with a triangulation $T_A:0\subseteq W_{1,A}\subseteq W_{2,A}\subseteq \cdots\subseteq W_{d,A}=W_A$ whose parameter is
 $\{\delta_{i,A}\}_{i=1}^d$ such that $[W_{i,A}]\in \mathrm{H}^1(G_K, W_{i-1,A}(\delta_{i,A}^{-1}))$ is a lift of 
 $[W_{i,A/I}]\in \mathrm{H}^1(G_K, W_{i-1,A/I}(\delta_{i,A/I}^{-1}))$ for any $1\leqq i\leqq d$. This shows that $f_{z}$ is formally smooth. 
Because this property is preserved by any base change of the base field $E$, $f$ is smooth at $z$ by 
Proposition 2.9 of \cite{BLR95}.
 
 \end{proof}
 
We next prove a proposition which can be seen as a local analogue of the classicality theorem 
of overconvergent modular forms. Let $(V_x,T_x)$ be a pair as in the above proposition such that $z_{(V_x,T_x)}$ is an $E$-rational point 
 of $\mathcal{E}_{\overline{V}}$. 
 We take an 
 affinoid open neighborhood $U=\mathrm{Spm}(R)\subseteq \mathcal{E}_{\overline{V}}$ of $z_{(V_x,T_x)}$ which is $Y_i$-small 
 for any $1\leqq i\leqq d$, where we define $Y_d:=\mathrm{det}_R(N|_U)(\mathrm{rec}_K(\pi_K))
 \cdot Y_{[d-1]}^{-1}$. Then, for any $1\leqq i\leqq d$, the valuation $v_i:=v_p(\delta'_i(\pi_K))$ 
 is independent of $z'=(V',\delta'_1,\cdots,\delta'_{d-1})\in U$. 
 Take a sufficiently large $k\in\mathbb{Z}_{\geqq 1}$ satisfying the conditions (1), (2),
 \begin{itemize}
 \item[(1)]for any $1\leqq i\leqq d-1$ and $\sigma\in \mathcal{P}$, there exists a short exact sequence
 $$0\rightarrow (\bold{B}^+_{\mathrm{max},K}\hat{\otimes}_{K,\sigma}R)^{\varphi_K=Y_{[i]}}\rightarrow 
 \bold{B}^+_{\mathrm{dR}}/t^k\bold{B}^+_{\mathrm{dR}}\hat{\otimes}_{K,\sigma}R \rightarrow U_{i,\sigma}\rightarrow 0$$ 
 of Banach $R$-modules with the property (Pr)
 for a Banach $R$-module $U_{i,\sigma}$. 
 \item[(2)]$k> \frac{\mathrm{max}\{2, (d-1)^2\}}{f} \mathrm{max}_{1\leqq i\leqq d}\{ |v_i|\}$.
 
 \end{itemize}
Fix $k$ which satisfies the above conditions. We define a subset $\mathcal{W}^{\times d}_{k}$ of $\mathcal{W}^{\times d}$ by 
\begin{multline*}
\mathcal{W}^{\times d}_{k}:=\{(\prod_{\sigma\in\mathcal{P}}\sigma^{k'_{1,\sigma}}, \prod_{\sigma\in \mathcal{P}}\sigma^{k'_{2,\sigma}},
 \cdots, \prod_{\sigma\in \mathcal{P}}\sigma^{k'_{d,\sigma}}) \in \mathcal{W}^{\times d} | k'_{i,\sigma}\in \mathbb{Z},\\
   k'_{i,\sigma}-k'_{i+1,\sigma}>k \text{ for any } 
 1\leqq i\leqq d-1\text{ and } \sigma\in \mathcal{P} \}.
 \end{multline*}

Under this definition, we prove the following proposition which can be seen as a local analogue 
of the classicality theorem of overconvergent modular forms, which is crucial to prove the 
density of the crystalline points in $\mathcal{E}_{\overline{V}}$ (see Theorem \ref{1.13} below).

\begin{prop}\label{1.10}
Under the above situation, for any $z':=(V',\delta'_1,\cdots,\delta'_{d-1})\in U\cap f^{-1}(\mathcal{W}^{\times d}_{k})$, $V'$ is a crystalline and split trianguline $E(z')$-representation with a triangulation $T'$ whose parameter is 
$\{\delta'_i\}_{i=1}^d$, 
i.e. $z'=z_{(V', \mathcal{T'})}$.

\end{prop}
\begin{proof}
First, by the definition of $U$ and by Proposition \ref{1.2}, we have an isomorphism
$$(\bold{B}^+_{\mathrm{max},K}\hat{\otimes}_{K,\sigma}(N_i\otimes_{\mathcal{O}_{\mathcal{Z}_0}}R))^{G_K,\varphi_K=Y_{[i]}}
\isom (\bold{B}^+_{\mathrm{dR}}/t^k\bold{B}^+_{\mathrm{dR}}\hat{\otimes}_{K,\sigma}(N_i\otimes_{\mathcal{O}_{\mathcal{Z}_0}}R))^{G_K}$$ 
for any $1\leqq i\leqq d-1$ and $\sigma\in \mathcal{P}$. Because we have 
$f(z')\in \mathcal{W}^{\times d}_{k}$, we have $Q_{i,\sigma}(j)(z)\not= 0$ 
for any $\sigma\in\mathcal{P}$, $1\leqq i\leqq d-1$ and $-k\leqq j\leqq 0$. Hence, by Corollary 2.6 of  \cite{Ki03}, the natural base change map is an isomorphism
$$(\bold{B}^+_{\mathrm{dR}}/t^k\bold{B}^+_{\mathrm{dR}}\hat{\otimes}_{K,\sigma}(N_i\otimes_{\mathcal{O}_{\mathcal{Z}_0}}R))^{G_K}\otimes_R E(z')\isom 
(\bold{B}^+_{\mathrm{dR}}/t^k\bold{B}^+_{\mathrm{dR}}\otimes_{K,\sigma}(\wedge^iV')(\widetilde{\delta}^{' -1}_{[i]}))^{G_K}$$ 
 between one-dimensional vector spaces over $E(z')$ for each $\sigma$ and $i$. The natural map $$(\bold{B}^+_{\mathrm{max},K}\otimes_{K,\sigma}E(z'))^{\varphi_K=
\delta^{' }_{[i]}(\pi_K)}\hookrightarrow \bold{B}^+_{\mathrm{dR}}/t^k\bold{B}^+_{\mathrm{dR}}\otimes_{K,\sigma}E(z')$$ is an injection by the definition of $U$.
From these facts, the natural map 
$$(\bold{B}^+_{\mathrm{max},K}\hat{\otimes}_{K,\sigma}(\wedge^iV')(\widetilde{\delta}^{' -1}_{[i]}))^{G_K, \varphi_K=
\delta'_{[i]}(\pi_K)}\isom (\bold{B}^+_{\mathrm{dR}}/t^k\bold{B}^+_{\mathrm{dR}}\otimes_{K,\sigma}(\wedge^iV')(\widetilde{\delta}^{' -1}_{[i]}))^{G_K}$$ 
is an isomorphism for each $\sigma\in \mathcal{P}$ and $1\leqq i\leqq d-1$. From this isomorphism and because  we have $f(z')\in \mathcal{W}^{\times d}_k$, 
we can check that, for each $1\leqq i\leqq d-1$,
$$D_i:=\bold{D}^+_{\mathrm{cris}}((\wedge^iV')(\widetilde{\delta}^{' -1}_{[i]}))^{\varphi^f=\delta'_{[i]}(\pi_K)}$$ is a sub rank one 
$E(z')$-filtered $\varphi$-module of $\bold{D}_{\mathrm{cris}}((\wedge^iV')(\widetilde{\delta}^{' -1}_{[i]}))$ such that $\mathrm{Fil}^0(K\otimes_{K_0}D_i)=
K\otimes_{K_0}D_i$ and $\mathrm{Fil}^1(K\otimes_{K_0}D_i)=0$. Hence, by Lemma 2.21 of \cite{Na10}, we obtain a saturated injection 
$$W(\delta_{(\delta'_{[i]}(\pi_K))})\hookrightarrow W((\wedge^iV')(\widetilde{\delta}^{' -1}_{[i]})) $$ 
and, twisting by $\widetilde{\delta}'_{[i]}$, we also obtain a following saturated injection
$$W(\delta'_{[i]})\hookrightarrow W(\wedge^iV')$$ for each $1\leqq i\leqq d-1$. 
By Proposition \ref{1.11} below, then $V'$ is a split trianguline $E(z')$-representation with a triangulation 
$T':0\subseteq W'_1\subseteq W'_2\subseteq \cdots \subseteq W'_d=W(V')$ 
whose parameter is equal to $\{\delta'_i\}_{i=1}^d$. To finish the proof, it suffices to show that $W'_i$ is crystalline for any 
$1\leqq i\leqq d$ by induction on $i$.
 By Lemma \ref{1.12} below, it suffices to check that the parameter $\{\delta'_i\}_{i=1}^d$ satisfies the conditions (1) and (2) in this lemma.
For (1), it is trivial by the definition of $\mathcal{W}^{\times d}_{k}$. For (2), if 
$\delta'_i/\delta'_j=\prod_{\sigma\in \mathcal{P}}\sigma^{k_{\sigma}}
|N_{K/\mathbb{Q}_p}|_p$ for some $\{k_{\sigma}\}_{\sigma\in \mathcal{P}}\in \prod_{\sigma\in \mathcal{P}}\mathbb{Z}_{\geqq 1}$ and for some $1\leqq i<j\leqq d$, then 
$k_{\sigma}=k_{i,\sigma}-k_{j,\sigma}\geqq k+1$ for any $\sigma\in\mathcal{P}$. Hence the slope 
of $W(\delta'_i/\delta'_j)$ is $\frac{1}{[K:\mathbb{Q}_p]}(\sum_{\sigma\in \mathcal{P}}k_{\sigma})-1\geqq k$. On the other hands, 
the slope of $W(\delta'_i/\delta'_j)$ can be computed by $\frac{1}{f}(v_p(\delta'_i(\pi_K))-v_p(\delta'_j(\pi_K)))< k$, which is a contradiction.
Hence $\{\delta'_i\}_{i=1}^d$ satisfies (1), (2) of Lemma \ref{1.12}, hence $V'$ is crystalline.

\end{proof}

\begin{prop}\label{1.11}
Let $z:=(V,\delta_1,\cdots,\delta_{d-1})\in\mathcal{Z}_0$ be a point which satisfies 
the following conditions $(1)$ and $(2)$.
\begin{itemize}
\item[(1)]For each $1\leqq i\leqq d-1$, there exists a saturated injection 
$$W(\delta_{[i]})\hookrightarrow W(\wedge^i V).$$
\item[(2)]One of the following conditions holds.
\begin{itemize}
\item[(i)]
For any $1\leqq i\leqq d-1$ and for any $J\not=[i]\subseteq [d]$ such that $\sharp(J)=i$,
$k(\delta_{[i]})_{\sigma}-k(\delta_{J})_{\sigma}\not\in \mathbb{Z}_{\geqq 0}$ for any $\sigma\in\mathcal{P}$.
\item[(ii)]For any $1\leqq i\leqq d$ and $\sigma\in\mathcal{P}$, $k(\delta_{i})_{\sigma}$ is an integer,  and , if we put 
$v_0:=\mathrm{max}_{1\leqq i\leqq d}\{|v_p(\delta_i(\pi_K))|\}$, $k(\delta_{i})_{\sigma}-k(\delta_{i+1})_{\sigma}> \frac{(d-1)^2}{f}v_0$ for 
any $1\leqq i \leqq d-1$ and $\sigma\in \mathcal{P}$,

\end{itemize}

\end{itemize}
Then $V$ is a split trianguline $E(z)$-representation with 
a triangulation $T$ whose parameter is $\{\delta_i\}_{i=1}^d$, i.e. $z=z_{(V,T)}$.

\end{prop}
\begin{proof}
First, by the condition (1) for $i=1$, we have a saturated injection 
$W(\delta_1)\hookrightarrow W(V)$. We denote by $W_1\subseteq W(V)$ the image of this 
injection. By induction on $i$, we show that we can take a filtration 
$0\subseteq W_1\subseteq W_2\subseteq \cdots W_{i-1}\subseteq W_i\subseteq W(V)$ such that 
$W_i$ is a $E(z)$-$B$-pair of rank $i$ which is saturated in $W(V)$ and 
$W_{i}/W_{i-1}\isom W(\delta_{i})$ and $\wedge^iW_i\subseteq W(\wedge^iV)$ is equal to the 
image of the given injection $W(\delta_{[i]})\hookrightarrow W(\wedge^iV)$ in (1).
 We assume that we can take $0\subseteq W_1\subseteq \cdots \subseteq W_{i-1}\subseteq W(V)$ satisfying all the above 
 conditions. Denote by $W'$ the cokernel of the injection $W_{i-1}\subseteq W(V)$. If we take the $i$-th exterior product, 
 we obtain a following short exact sequence of $E(z)$-$B$-pairs
 $$0\rightarrow W(\delta_{[i-1]})\otimes W'\rightarrow W(\wedge^iV)\rightarrow W''\rightarrow 0$$
  for a $E(z)$-$B$-pair $W''$ because we have a natural isomorphism
 $\wedge^{i-1}W_{i-1}\isom W(\delta_{[i-1]})$. Then, $W''$ is a successive extension of 
 $\wedge^jW_{i-1}\otimes \wedge^{i-j}W'$ for $0\leqq j\leqq i-2$. 
We define a map $\iota:W(\delta_{[i]})\rightarrow 
 W''$ as the composition of the injection $W(\delta_{[i]})\hookrightarrow W(\wedge^iV)$ in $(1)$ with the canonical surjection $W(\wedge^i V)\twoheadrightarrow W''$. Under this situation, we claim that the map $\iota:W(\delta_{[i]})\rightarrow 
 W''$ is zero under the condition (2). 
 
 We first prove this claim under the condition (i) of (2). 
 In this case, if $\iota$ is not zero, then this is an injection because $W(\delta_{[i]})$ is of rank one. 
 By Proposition 2.14 of \cite{Na09}, the saturation of the image of $\iota$ is isomorphic to $W(\delta_{[i]}\prod_{\sigma\in \mathcal{P}}
 \sigma^{-k_{\sigma}})$ for some $\{k_{\sigma}\}_{\sigma\in \mathcal{P}}\in \prod_{\sigma\in \mathcal{P}}\mathbb{Z}_{\geqq 0}$. 
 This implies that $W''$ has Hodge-Tate weights $\{k(\delta_{[i]})_{\sigma} -k_{\sigma}\}_{\sigma\in \mathcal{P}}$. However, because we have $z\in\mathcal{Z}_0$ 
 and the set of Hodge-Tate weights of $W_{i-1}$ is $\{k(\delta_{1})_{\sigma},\cdots,k(\delta_{i-1})_{\sigma}\}_{\sigma\in \mathcal{P}}$,  each $\sigma$-part of the Hodge-Tate weights of 
 $W''$ is equal to $k(\delta_{J})_{\sigma}$ for some $J\not=[i]$ such that $\sharp(J)=i$, which  contradicts to 
 the condition (i), hence the map $\iota$ must be zero.
 
 We next prove the claim under the condition (ii). We assume that the map $\iota$ is not zero. Because 
 $W''$ is a successive extension of 
 $\wedge^jW_{i-1}\otimes \wedge^{i-j}W'$ for $0\leqq j\leqq i-2$, we obtain 
 a saturated injection $$W(\delta_{[i]}\prod_{\sigma\in \mathcal{P}}\sigma^{-k_{\sigma}})\hookrightarrow \wedge^jW_{i-1}\otimes \wedge^{i-j}W'$$ for some $0\leqq j\leqq i-2$ and for some 
 $\{k_{\sigma}\}_{\sigma\in \mathcal{P}}\in \prod_{\sigma\in \mathcal{P}}\mathbb{Z}_{\geqq 0}$. Because, each
 $\sigma$-part of the Hodge-Tate weight of $W''$ is equal to $k(\delta_{J})_{\sigma}$ for some 
 $J\not=[i]$ such that $\sharp(J)=i$, there exists such a $J$ such that 
 $$k_{\sigma}=k(\delta_{[i]})_{\sigma}-k(\delta_{J})_{\sigma}>\frac{(d-1)^2}{f}v_0$$ by the condition (ii). Because the slope of $W(\delta_{[i]}
 \prod_{\sigma\in \mathcal{P}}\sigma^{-k_{\sigma}})$ is equal to $$\frac{1}{f}v_p(\delta_{[i]}(\pi_K)) -\frac{1}{[K:\mathbb{Q}_p]}
 \sum_{\sigma\in \mathcal{P}}k_{\sigma}< \frac{i}{f}v_0-\frac{(d-1)^2}{f}v_0= \frac{i-(d-1)^2}{f}v_0,$$ 
so the smallest slope, which we denote by $s''$, of $ \wedge^jW_{i-1}\otimes \wedge^{i-j}W'$ satisfies $$s''< \frac{i-(d-1)^2}{f}v_0.$$ 
 On the other hands, because we have an injection $W(\delta_{[i-1]})\otimes W'\hookrightarrow W(\wedge^iV)$ and $W(\wedge^iV)$ is \'etale, 
 the smallest slope $s'$ of $W'$ satisfies 
 $$s'\geqq -\frac{1}{f}v_p(\delta_{[i-1]}(\pi_K))\geqq -\frac{(i-1)}{f}v_0$$ by Corollary 1.6.9 
 of \cite{Ke08}. Because all the slopes of $W_{i-1}$ are positive or zero by Corollary 1.6.9 of \cite{Ke08}, 
 then the smallest slope  
 $s''$ of $ \wedge^jW_{i-1}\otimes \wedge^{i-j}W'$ satisfies that 
 $$s''\geqq \mathrm{min}\{is', 2s'\}\geqq -\frac{(i-1)i}{f}v_0$$ by Remark 1.7.2 of \cite{Ke08}. Hence, we obtain an inequality $-\frac{(i-1)i}{f}v_0\leqq s'' < \frac{i-(d-1)^2}{f}v_0$, which implies that 
 $(d-1)^2<i^2$, this is a contradiction, hence the map $\iota$ must be zero. We finish to 
 prove the claim in both cases.
 
 This claim implies that the given injection $W(\delta_{[i]})\hookrightarrow W(\wedge^iV)$ factors through a saturated injection 
 $$W(\delta_{[i]})\hookrightarrow W(\delta_{[i-1]})\otimes W'\hookrightarrow W(\wedge^iV).$$ Twisting the first injection by 
 $\delta_{[i-1]}^{-1}$, we obtain a saturated injection $W(\delta_i)\hookrightarrow W'$. If we denote by $W_i\subseteq W(V)$ the inverse image of 
 $W(\delta_i)\subseteq W'$ by the canonical surjection $W(V)\twoheadrightarrow W'$, we obtain a short exact sequence
 $$0\rightarrow W_{i-1}\rightarrow W_i\rightarrow W(\delta_i)\rightarrow 0,$$ and 
 we can check that $W_i$ satisfies the desired properties. 
 By induction, we finish to prove the proposition.

\end{proof}
\begin{lemma}\label{1.12}
Let $W$ be a split trianguline $E$-$B$-pair  of rank $d$ with a triangulation 
$T:0\subseteq W_1\subseteq \cdots \subseteq W_{d-1}\subseteq W_d=W$ with the parameter $\{\delta_i\}_{i=1}^d$. If $\{\delta_i\}_{i=1}^d$ 
satisfies the following conditions (1) and (2), 
\begin{itemize}
\item[(1)]for any $1\leqq i\leqq d$, $\delta_i|_{\mathcal{O}_K^{\times}}=
\prod_{\sigma\in \mathcal{P}}\sigma^{k_{i,\sigma}}$ for some $\{k_{i,\sigma}\}_{\sigma\in \mathcal{P}}
\in \prod_{\sigma\in \mathcal{P}}\mathbb{Z}$ such that $k_{1,\sigma}>k_{2,\sigma}>\cdots>k_{d,\sigma}$ for any  $\sigma\in \mathcal{P}$,
\item[(2)]for any $1\leqq i<j\leqq d$, $\delta_i/\delta_j\not= \prod_{\sigma\in \mathcal{P}}\sigma^{k_{\sigma}}|N_{K/\mathbb{Q}_p}|_p$ 
for any $\{k_{\sigma}\}_{\sigma\in \mathcal{P}}\in \prod_{\sigma\in \mathcal{P}}\mathbb{Z}_{\geqq 1}$,

\end{itemize}
then $W$ is crystalline.

\end{lemma}
 \begin{proof}
 We prove this lemma by induction on the rank of $W$. If $W$ is of rank one, the condition 
 (1) implies that $W=W(\delta_1)$ is crystalline.
 We assume that $W$ is of rank $d$ and $W_{d-1}$ is crystalline. For any $B$-pair $W'$,
 let $$\mathrm{H}^1_f(G_K,W'):=\mathrm{Ker}(\mathrm{H}^1(G_K, W')\rightarrow \mathrm{H}^1(G_K, W'_e\otimes_{\bold{B}_e}\bold{B}_{\mathrm{cris}}))$$ be the Bloch-Kato's finite cohomology of $W'$ 
 defined in Definition 2.4 of \cite{Na09}.
 We claim that the natural injection 
 $$\mathrm{H}^1_f(G_K, W_{d-1}(\delta_d^{-1}))\hookrightarrow \mathrm{H}^1(G_K, W_{d-1}(\delta_d^{-1}))$$ is 
 a bijection, which proves that $W$ is crystalline. 
 We prove this claim by computing the dimensions of both $E$-vector spaces.
 First, we have $\mathrm{H}^2(G_K, W(\delta_i/\delta_d))=0$  for any $1\leqq i\leqq d-1$ by 
 the condition (2) and Proposition 2.9
 of \cite{Na10}. Because $W_{d-1}(\delta_d^{-1})$ is a successive extension of $W(\delta_i/\delta_d)$, we also 
 have $\mathrm{H}^2(G_K, W_{d-1}(\delta_d^{-1}))=0$. Hence, we obtain an equality 
  $$\mathrm{dim}_E\mathrm{H}^1(G_K, W_{d-1}(\delta_d^{-1}))
 =[K:\mathbb{Q}_p](d-1)+\mathrm{dim}_E\mathrm{H}^0(G_K, W_{d-1}(\delta_d^{-1}))$$ by Euler-Poincar\'e characteristic formula (Theorem 2.8 of \cite{Na10}).
 On the other hands, because $W_{d-1}(\delta_d^{-1})$ is crystalline, we have an equality
\begin{multline*}
\mathrm{dim}_E\mathrm{H}^1_f(G_K, W_{d-1}(\delta_d^{-1}))
 =\mathrm{dim}_E\bold{D}_{\mathrm{dR}}(W_{d-1}(\delta_d^{-1}))/\mathrm{Fil}^0\bold{D}_{\mathrm{dR}}(W_{d-1}(\delta_d^{-1}))\\
 +\mathrm{dim}_E\mathrm{H}^0(G_K,W_{d-1}(\delta_d^{-1}))
  \end{multline*}
by Proposition 2.7 of \cite{Na09}. Because $W_{d-1}(\delta_{d}^{-1})$ is a successive extension of 
 $W(\delta_i/\delta_d)$, the condition (1) implies that $\mathrm{Fil}^0\bold{D}_{\mathrm{dR}}(W_{d-1}(\delta_d^{-1}))=0$. Therefore, we obtain the following equalities
 \begin{multline*}
\mathrm{dim}_E\mathrm{H}^1_f(G_K, W_{d-1}(\delta_d^{-1}))=\mathrm{dim}_E\bold{D}_{\mathrm{dR}}(W_{d-1}(\delta_d^{-1}))
 +\mathrm{dim}_E\mathrm{H}^0(G_K,W_{d-1}(\delta_d^{-1}))\\
 =[K:\mathbb{Q}_p](d-1)+\mathrm{dim}_E\mathrm{H}^0(G_K,W_{d-1}(\delta_d^{-1}))
 =\mathrm{dim}_E\mathrm{H}^1(G_K, W_{d-1}(\delta_d^{-1})).
  \end{multline*}
We finish to prove the claim, hence we finish to prove the lemma.

 \end{proof}
 
 We define two subsets $\mathcal{X}_{\overline{V},\mathrm{reg-\mathrm{cris}}}$ and 
 $\mathcal{X}_{\overline{V},b}$ of $\mathcal{X}_{\overline{V}}$ by 

\begin{multline*}
\mathcal{X}_{\overline{V},\mathrm{reg-\mathrm{cris}}}:=\{x=[V_x]\in \mathcal{X}_{\overline{V}}|V_x \text{ is crystalline with Hodge-Tate weights }\\
\{k_{i,\sigma}\}_{1\leqq i\leqq d,\sigma\in \mathcal{P}} 
 \text{ such that  }  k_{i,\sigma}\not=k_{j,\sigma} 
 \text{ for any } i\not=j \text{ and }\sigma\in \mathcal{P}\}
   \end{multline*}
\begin{multline*}
\mathcal{X}_{\overline{V},b}:=\{x=[V_x]\in \mathcal{X}_{\overline{V}}| V_x\otimes_{E(x)}E' \text{ is benign for a finite extension }
E' \text{of } E(x) \\
 \text{ and } T_{\tau} \text{ satisfies the condition (1) of Theorem \ref{1.7} for any } \tau\in \mathfrak{S}_d\}
\end{multline*}

Using the propositions proved in this subsection, we can prove the following theorem, which states the Zariski density of 
 crystalline points in $\mathcal{E}_{\overline{V}}$.

  \begin{thm}\label{1.13}
  Let $z_{(V_x,T_x)}\in \mathcal{E}_{\overline{V}}$ be an $E$-rational  point satisfying all the conditions in Theorem \ref{1.7}.
  Then, for any admissible open neighborhood $U\subseteq \mathcal{E}_{\overline{V}}$ of $z_{(V_x,T_x)}$, there exists 
  a smaller admissible open neighborhood $U'\subseteq U$ of $z_{(V_x,T_x)}$ such that the subset defined by 
  $$U'_{\mathrm{cris}}:=\{z=([V],\delta_1,\cdots,\delta_{d-1})\in U'| [V]\in \mathcal{X}_{\overline{V},\mathrm{reg}-\mathrm{cris}} \}$$ 
  is Zariski dense in $U'$.

 \end{thm}
 \begin{proof}
 If we use Proposition \ref{1.9} and Proposition \ref{1.10}, the proof of this theorem is same as that of Lemma 4.7 of 
 \cite{Na10}.
 \end{proof}
 

\section{Zariski density of crystalline representations for any $p$-adic field}
 In this final chapter, we prove the main theorems of this article.

 \begin{lemma}\label{1.14}
 Let $x=[V_x]\in \mathcal{X}_{\overline{V},\mathrm{reg-\mathrm{cris}}}$ be a point. 
 Then, for any admissible open neighborhood $U\subseteq \mathcal{X}_{\overline{V}}$ of $x$, 
 $U\cap \mathcal{X}_{\overline{V},b}$ is not empty.
 
 \end{lemma}
 \begin{proof}
 This lemma is a generalization of Lemma 4.12 of \cite{Na10}.
 We may assume that $E(x)=E$ and that 
 the set of Hodge-Tate weights $\tau:=
 \{k_{i,\sigma}\}_{1\leqq i\leqq d, \sigma\in \mathcal{P}}$  
 of $V_x$ satisfies $k_{1,\sigma}>k_{2,\sigma}>\cdots>k_{d,\sigma}$ for any $\sigma\in \mathcal{P}$. 
 By Corollary 2.7.7 of \cite{Ki08}, the subset 
 $\mathcal{X}_{\overline{V},\mathrm{cris}}^{\tau}$ of $\mathcal{X}_{\overline{V}}$ consisting of 
 the points corresponding to crystalline representations with Hodge-Tate weights
 $\{k_{i,\sigma}\}_{1\leqq i\leqq d,\sigma\in \mathcal{P}}$ forms a Zariski closed subspace of $\mathcal{X}_{\overline{V}}$ 
 corresponding to a quotient $R_{\overline{V},\mathrm{cris}}^{\tau}$ of $R_{\overline{V}}$. We consider 
 the universal framed deformation ring $R^{\square}_{\overline{V}}$ of ($\overline{V},\beta$), where $\beta$ is a fixed 
 $\mathbb{F}$-base of $\overline{V}$. Then, in the same way as $R^{\tau}_{\overline{V},\mathrm{cris}}$, we obtain 
 a quotient $R^{\square,\tau}_{\overline{V},\mathrm{cris}}$ of $R^{\square}_{\overline{V}}$ and we have a natural
 map $R^{\tau}_{\overline{V},\mathrm{cris}}\rightarrow R^{\square,\tau}_{\overline{V},\mathrm{cris}}$ which is induced from the map $R_{\overline{V}}\rightarrow R^{\square}_{\overline{V}}$ corresponding 
 to the forgetting map $D^{\square}_{\overline{V}}(A)\rightarrow D_{\overline{V}}(A):(V_A,\psi_A,\widetilde{\beta})\mapsto 
 (V_A,\psi_A)$ ($A\in \mathcal{C}_{\mathcal{O}}$), where $\widetilde{\beta}$ is an $A$-base of $V_A$ which is a
  lift of $\beta$. Therefore, if we denote by $\mathcal{X}^{\square,\tau}_{\overline{V},\mathrm{cris}}$ the rigid analytic space 
  associated to $R^{\square,\tau}_{\overline{V},\mathrm{cris}}$, it suffices to show the following lemma.
  
 \end{proof}
 \begin{lemma}\label{1.15}
 Let $x$ be a point of $\mathcal{X}^{\square,\tau}_{\overline{V},\mathrm{cris}}$ and 
 let $U$ be an admissible open neighborhood of $x$, then there exists a point $z$ of $U$
 whose corresponding representation is benign and satisfies the condition $(1)$ of Theorem \ref{1.7}.

 \end{lemma}

 \begin{proof}
 We remark that, by the proof of 
Theorem 3.3.8 of \cite{Ki08}, 
we have a natural isomorphism $\hat{\mathcal{O}}_{\mathcal{X}_{\overline{V}, \mathrm{cris}}^{\square,\tau},y}\isom 
R_{V_y}^{\square,\mathrm{cris}}$ for each $y\in \mathcal{X}_{\overline{V}, \mathrm{cris}}^{\square,\tau}$. By Corollary 6.3.3 of \cite{Be-Co08} and by Corollary 3.19 of \cite{Ch09}, 
there exists an admissible affinoid open neighborhood $U=\mathrm{Spm}(R)$ of $x$ in $\mathcal{X}_{\overline{V},\mathrm{cris}}^{\square,\tau}$, 
such that $$\bold{D}_{\mathrm{cris}}(V_R):= ((R\hat{\otimes}_{\mathbb{Q}_p}\bold{B}_{\mathrm{cris}})\otimes_{R}V_R)^{G_K}$$
is a finite free $K_0\otimes_{\mathbb{Q}_p}R$-module of rank $d$ and $\bold{D}_{\mathrm{dR}}(V_R)\isom K\otimes_{K_0}\bold{D}_{\mathrm{cris}}(V_R)$ which are compatible with base changes, where $V_R$ is the restriction to $U$ of the universal deformation of $\overline{V}$.  
For each $\sigma'\in \mathrm{Gal}(K_0/\mathbb{Q}_p)$, we denote  by  $D_{\sigma'}$ the $\sigma'$-component of $\bold{D}_{\mathrm{cris}}(V_R)$. We denote by 
$$T^d+a_{d-1}T^{d-1}+\cdots+a_1T+a_0:=\mathrm{det}_R(T\cdot\mathrm{id}_{D_{\sigma'}}-\varphi^f|_{D_{\sigma'}})
\in R[T]$$ the characteristic polynomial of the relative Frobenius on $D_{\sigma'}$ , which is independent of $\sigma'\in \mathrm{Gal}(K_0/\mathbb{Q}_p)$. 
Denote by $\Delta\in R$ the discriminant of this polynomial.
Then, we claim that $\Delta$ is a non zero divisor of $R$, i.e. the subset $U_{\Delta}\subseteq U$ consisting of the points 
$z$ such that $\bold{D}_{\mathrm{cris}}(V_z)$ have $d$-distinct relative Frobenius eigenvalues is scheme theoretically dense 
in $U$. To prove this claim, it suffices to show that $\Delta\not= 0$ in $\hat{\mathcal{O}}_{\mathcal{X}(\bar{\rho})_{\mathrm{cris}}^{\square, \tau},z}\isom R_{V_z}^{\square, \mathrm{cris}}$ for any $z\in U$ because 
$R_{V_z}^{\square, \mathrm{cris}}$ is domain by Theorem 3.3.8 of \cite{Ki08}. 
It is easy to see that $\bold{D}_{\mathrm{cris}}(V_z)$ can be deformed over $E(z)[\varepsilon]$ 
with $d$-distinct relative Frobenius 
eigenvalues, hence $\Delta \not =0 $ in $R_{V_z}^{\square, \mathrm{cris}}$. In the same way, we can show that the subset $U''\subseteq U$ consisting of the points $z$ such that $\bold{D}_{\mathrm{cris}}(V_z)$ have relative Frobenius eigenvalues $\{\alpha_i\}_{1\leqq i\leqq d}$ 
satisfying $\alpha_i\not= p^{\pm f}\alpha_j$ for any $i\not=j$  is also scheme theoretically dense in $U$. Hence, their intersection 
$U_{\Delta}\cap U''$ is also scheme theoretically dense in $U$. 
Take an element $z\in U_{\Delta}\cap U''\subseteq U$. Extending scalars, 
we may assume that $$\bold{D}_{\mathrm{cris}}(V_z)=\oplus_{i=1}^dK_0\otimes_{\mathbb{Q}_p}Ee_{i,z}$$ such that 
$\varphi^f(e_{i,z})=\alpha_{i,z}e_{i,z}$ for some $\alpha_{i,z}\in E^{\times}$ ($1\leqq i\leqq d$) such that $\alpha_{i,z}\not=\alpha_{j,z},p^{\pm f}\alpha_{j,z}$ for any $i\not= j$. Because 
$\mathcal{O}_{\mathcal{X}(\bar{\rho})^{\square, \tau}_{\mathrm{cris}},z}$ is Henselian by Theorem 2.1.5 of \cite{Berk93}, if we 
take a sufficiently small affinoid open 
neighborhood $U'=\mathrm{Spm}(R')$ of $z$ in $U_{\Delta}\cap U''$, then we can write
$$\bold{D}_{\mathrm{cris}}(V_R)\otimes_RR'= \oplus_{i=1}^dK_0\otimes_{\mathbb{Q}_p}R'e_i$$
such that $K_0\otimes_{\mathbb{Q}_p}R'e_i$ is $\varphi$-stable and $\varphi^f(e_i)=\widetilde{\alpha}_ie_i$ for some $\widetilde{\alpha}_i\in R^{' \times}$ for $1\leqq i\leqq d$ 
satisfying that $\widetilde{\alpha}_i-\widetilde{\alpha}_j, \widetilde{\alpha}_i-p^{\pm f}\widetilde{\alpha}_j\in R^{' \times}$ for any $i\not= j$. By Lemma 2.6.1 and by the proof of Corollary 2.6.2 of \cite{Ki08}, 
for sufficiently small $U'$,  if we decompose $\bold{D}_{\mathrm{dR}}(V_R)\otimes_{R}R'$ into $\sigma$-components by$$\bold{D}_{\mathrm{dR}}(V_R)\otimes_{R}R'\isom \bold{D}_{\mathrm{dR}}(V_{R'})=\oplus_{\sigma\in \mathcal{P}} D_{\sigma},$$ then, for each $\sigma\in \mathcal{P}$, the
$\sigma$-component $D_{\sigma}$ of $\bold{D}_{\mathrm{dR}}(V_{R'})$ is equipped with 
a filtration $\{\mathrm{Fil}^iD_{\sigma}\}_{i\in\mathbb{Z}}$ by finite free $R'$-modules such that 
$$\mathrm{Fil}^iD_{\sigma}\otimes_{R'}B\isom \mathrm{Fil}^i\bold{D}_{\mathrm{dR}}(V_{R'}\otimes_{R'} B)_{\sigma}$$ for any 
local $R'$-algebra $B$ which is finite over $E$. From these facts, we can
obtain two $R'$-bases $\{e_{i,\sigma}\}_{i=1}^d$ and $\{f_{i,\sigma}\}_{i=1}^d$ of $D_{\sigma}$, where 
$\{e_{i,\sigma}\}_{i=1}^d$ is the basis naturally induced from the basis $\{e_{i}\}_{i=1}^d$ of $\bold{D}_{\mathrm{cris}}(V_{R'})$,
and $\{f_{i,\sigma}\}_{i=1}^d$ is a basis which satisfies that 
$$\mathrm{Fil}^{-k_{i,\sigma}}D_{\sigma}=R'f_{i,\sigma}\oplus R'f_{i+1,\sigma}\oplus 
\cdots\oplus R'f_{d,\sigma}$$ for any $1\leqq i\leqq d$. For each $\sigma\in \mathcal{P}$, we define a $(d\times d)$-matrix $A_{\sigma}:=(a_{i,j,\sigma})_{i,j}$ by 
$f_{j,\sigma}:=\sum_{i=1}^d a_{i,j,\sigma}e_{i,\sigma}$. We denote by $a\in R$ the product of all $k$-th minor determinants of $A_{\sigma}$ for 
all $1\leqq k\leqq d-1$ and $\sigma\in \mathcal{P}$. By the definition of benign representation, for any $z\in \mathrm{Spm}(R')$, it is 
easy to see that $V_z$ is benign if and only if $a(z)\not=0$ in $E(z)$.
Therefore, to prove the lemma, it suffices to show that 
$\mathrm{Spm}(R')_a$  and $\mathrm{Spm}(R')_{(\widetilde{\alpha}_{J_1}-\widetilde{\alpha}_{J_2})}$ for any $1\leqq i\leqq d-1$ and for any subsets $J_1\not=J_2$ of $[d]$ such that $\sharp(J_1)=\sharp(J_2)=i$ are all scheme theoretically dense in $\mathrm{Spm}(R')$, i.e. it suffices to show that 
both $a$ and $\widetilde{\alpha}_{J_1}-\widetilde{\alpha}_{J_2}$ are non zero divisors of $R'$. Because we have an isomorphism $\hat{R}'_{\mathfrak{m}_z}\isom R_{V_z}^{\square, \mathrm{cris}}$ and $R_{V_z}^{\square,\mathrm{cris}}$ is domain for any 
$z\in \mathrm{Spm}(R')$, it suffices to show that both $a$ and $\widetilde{\alpha}_{J_1}-\widetilde{\alpha}_{J_2}$ are non zero in $R_{V_z}^{\square, \mathrm{cris}}$. Finally, this claim can be easily proved by explicitly constructing lifts of the filtered $\varphi$-module $\bold{D}_{\mathrm{cris}}(V_z)$ over $E(z)[\varepsilon]$ such that the values corresponding to $a$ or $\widetilde{\alpha}_{J_1}-\widetilde{\alpha}_{J_2}$ for the lifts are non zero.

 \end{proof}
 
 For a rigid analytic space $Y$ over $E$ and for a point $y\in Y$, we denote the tangent space at $y$ 
 by 
$$t_{Y,y}:=\mathrm{Hom}_{E(y)}(\mathfrak{m}_y/\mathfrak{m}_y^2, E(y)),$$
 where 
$\mathfrak{m}_y$ is the maximal ideal of $\mathcal{O}_{Y,y}$. 

Denote by $\overline{\mathcal{X}}_{\overline{V},\mathrm{reg}-\mathrm{cris}}$ the Zariski closure of $\mathcal{X}_{\overline{V},\mathrm{reg}-\mathrm{cris}}$ in $\mathcal{X}_{\overline{V}}$. 
 The following theorems are the main theorems of this paper.
 
 \begin{thm}\label{1.16}
$\overline{\mathcal{X}}_{\overline{V},\mathrm{reg}-\mathrm{cris}}$ is a union of irreducible components of 
$\mathcal{X}_{\overline{V}}^{\mathrm{red}}$, and each irreducible component of $\mathcal{X}_{\overline{V}}^{\mathrm{red}}$ which is contained in $\overline{\mathcal{X}}_{\overline{V},\mathrm{reg}-\mathrm{cris}}$ is equidimensional of dimension $d^2[K:\mathbb{Q}_p]+1$.
 
 \end{thm}
 
 \begin{proof}
Let $Z$ be an irreducible component of $\overline{\mathcal{X}}_{\overline{V},\mathrm{reg}-\mathrm{cris}}$. Because the singular locus $Z_{\mathrm{sing}}\subseteq Z$ is a proper Zariski closed subset of $Z$, Lemma \ref{1.14} implies that there exists a point 
$x\in \mathcal{X}_{\overline{V},\mathrm{b}}\cap Z$ such that $Z$ is smooth at $x$. Moreover, $\mathcal{X}_{\overline{V}}$ is also smooth at $x$ of its dimension $d^2[K:\mathbb{Q}_p]+1$ by 
Corollary 2.50 of \cite{Na10}. Because any irreducible component is equidimensional by a remark 
in page.14 of \cite{Con99}, then it suffices to show that the natural inclusion 
$t_{Z,x}\hookrightarrow t_{\mathcal{X}_{\overline{V}},x}$ is an isomorphism.

By the definition of benign representation and by Proposition \ref{1.4}, the point $z_{(V_x,T_{\tau})}\in\mathcal{Z}_0$ 
corresponding to the pair $(V_x,T_{\tau})$ is contained in $\mathcal{E}_{\overline{V}}$ for each $\tau\in \mathfrak{S}_d$.
For each $\tau\in\mathfrak{S}_d$, we denote by $Y_{\tau}$ the irreducible component of 
$p^{-1}(\overline{\mathcal{X}}_{\overline{V},\mathrm{reg}-\mathrm{cris}})$ containing $z_{(V_x,T_{\tau})}$, which is uniquely determined and also an irreducible component of 
$\mathcal{E}_{\overline{V}}$ by Theorem \ref{1.7} and Theorem \ref{1.13}.
Because the natural morphism $p|_{Y_{\tau}}:Y_{\tau}\rightarrow \mathcal{X}_{\overline{V}}$ factors through $p|_{Y_{\tau}}:Y_{\tau}\rightarrow Z\hookrightarrow\mathcal{X}_{\overline{V}}$ for any $\tau\in \mathfrak{S}_d$, 
we obtain a map
$$t_{\mathcal{E}_{\overline{V}},z_{(V_x,T_{\tau})}}=t_{Y_{\tau}, z_{(V_x,T_{\tau})}}\rightarrow t_{Z,x}\hookrightarrow t_{\mathcal{X}_{\overline{V}}, x}$$
for each $\tau\in \mathfrak{S}_d$, where the first equality follows from Theorem \ref{1.7} and 
Theorem \ref{1.13}.
Summing up for all $\tau\in\mathfrak{S}_d$, we obtain a map 
$$\bigoplus_{\tau\in \mathfrak{S}_d}t_{\mathcal{E}_{\overline{V}}, z_{(V_x,T_{\tau})}}\rightarrow t_{Z,x}\hookrightarrow t_{\mathcal{X}_{\overline{V}}, x}.$$ 
By Theorem \ref{4.3} and Theorem \ref{1.7}, this map is surjective, hence we obtain an equality 
$$t_{Z,x}=t_{\mathcal{X}_{\overline{V}},x},$$ which proves the theorem.

 \end{proof}
 Let $\omega:G_K\rightarrow \mathbb{F}^{\times}$ be the mod $p$ cyclotomic character.  
 Set $\mathrm{ad}(\overline{V}):=\mathrm{End}_{\mathbb{F}}(\overline{V})$ and 
 $\mathrm{ad}(\overline{V})^0:=\mathrm{ad}(\overline{V})^{\mathrm{trace}=0}$.
 
 \begin{thm}\label{1.17}
 Assume that $\overline{V}$ satisfies the following conditions,
 \begin{itemize}
 \item[(0)]$\mathrm{End}_{\mathbb{F}[G_K]}(\overline{V})=\mathbb{F}$,
 \item[(1)]$\mathcal{X}_{\overline{V},\mathrm{reg}-\mathrm{cris}}$ is non empty,
 \item[(2)]$\mathrm{H}^0(G_K, \mathrm{ad}(\overline{V})^0(\omega))=0$,
 \item[(3)]$\zeta_p\not\in K^{\times}$, or $p\not| d$,
  \end{itemize}
  then, we have an equality $\overline{\mathcal{X}}_{\overline{V},\mathrm{reg}-\mathrm{cris}}=\mathcal{X}_{\overline{V}}$.
 \end{thm}
 \begin{proof}
 First, we prove the theorem when $\zeta_p\not\in K^{\times}$. It suffices to show that 
 $\mathcal{X}_{\overline{V}}$ is irreducible by Theorem \ref{1.16}. This claim follows from 
 the fact that $\mathrm{H}^2(G_K, \mathrm{ad}(\overline{V}))=0$, which follows from the condition (2) and 
 the fact that $\mathrm{H}^0(G_K, \mathbb{F}(\omega))=0$ when $\zeta_p\not\in K^{\times}$.

 Next, we prove the theorem when $p\not| d$. 
 Let $P$ be the sub group of $\mathcal{O}_K^{\times}$ consisting of all $p$-th power roots of unity.
 Let $p^n$ be the order of $P$, take $\zeta_{p^n}\in \mathcal{O}_{K}^{\times}$ a primitive 
 $p^n$-th roots of unity. Fix $\sigma_0\in \mathcal{P}$. For each $1\leqq i\leqq p^n-1$, we define a subfunctor $D_i$ of $D_{\overline{V}}$ by 
 $$D_i(A):=\{[V_A]\in D_{\overline{V}}(A)| \mathrm{det}(V_A)(\mathrm{rec}_{K}(\zeta_{p^n}))=\iota_A(\sigma_0(\zeta_{p^n}))^i\}$$ 
 for $A\in \mathcal{C}_{\mathcal{O}}$, where $\iota_A:\mathcal{O}\rightarrow A$ is the morphism which gives the $\mathcal{O}$-algebra 
 structure on $A$. In the same way as in the proof of Theorem 4.16 of \cite{Na10}, we can prove that, under the condition 
 (2), $D_i$ is representable by a quotient $R_i$ of $R_{\overline{V}}$ which is formally smooth over $\mathcal{O}$, and, if we denote by $\mathcal{X}_i$ the rigid analytic space associated to $R_i$, 
 we have 
 an equality as rigid space
 $$ \mathcal{X}_{\overline{V}}=\coprod_{0\leqq i\leqq p^n-1} \mathcal{X}_i$$
  and $\mathcal{X}_i$ is irreducible. By the condition (1) and by Theorem  \ref{1.16}, there exists 
  $i$ such that $\mathcal{X}_i\subseteq \overline{\mathcal{X}}_{\overline{V},\mathrm{reg}-\mathrm{cris}}$. 
  Because we have $\chi_{\mathrm{LT}}^{p^f-1}\equiv 1 (\text{mod } \pi_K)$ and $\chi_{\mathrm{LT}}(\zeta_{p^n})=\zeta_{p^n}$, 
  twisting by $(\sigma_0\circ\chi_{\mathrm{LT}})^{(p^f-1)m}$ for each $m\in\mathbb{Z}$ induces an isomorphism $\mathcal{X}_{i}\isom \mathcal{X}_{i_m}$
  for $0\leqq i_m\leqq p^n-1$ such that $i_m\equiv i+(p^f-1)dm$ (mod $p^n$). Because $\chi_{\mathrm{LT}}$ is crystalline, 
  this isomorphism implies that $\mathcal{X}_{i_m} \subseteq \overline{\mathcal{X}}_{\overline{V},\mathrm{reg}-\mathrm{cris}}$. Under the assumption that $p\not| d$, each $0\leqq j \leqq p^n-1$ is equal to $i_m$ for some 
  $m\in \mathbb{Z}$, hence 
  we obtain an equality $\overline{\mathcal{X}}_{\overline{V},\mathrm{reg}-\mathrm{cris}}=\mathcal{X}_{\overline{V}}$.

 \end{proof}
 Finally, when $\overline{V}$ is absolutely irreducible, we obtain 
 the following corollary, which is a generalization of Theorem $A$ of \cite{Ch13} for general $K$.
 \begin{corollary}\label{1.18}
Assume that 
 \begin{itemize}
 \item[(1)]$\overline{V}$ is absolutely irreducible,
 \item[(2)]one of the following conditions holds,
 \begin{itemize}
 \item[(i)]
 $p\not|d$ and $\zeta_p\in K$,
 \item[(ii)]$\overline{V}\not\isom \overline{V}(\omega)$, 

 \end{itemize}
 \end{itemize}
 then we have an equality $\overline{\mathcal{X}}_{\overline{V},\mathrm{reg}-\mathrm{cris}}=\mathcal{X}_{\overline{V}}$.
 
 \end{corollary}
 \begin{proof}
 By Theorem \ref{1.17}, it suffices to show that the condition (1), (2) of the theorem implies 
 the conditions (2), (3) of Theorem \ref{1.17}.
 
 We first claim that, under the assumption (1), the condition (2) implies (in fact is equivalent to) that 
 $\mathrm{H}^0(G_K,\mathrm{ad}(\overline{V})^0(\omega))=0$. This claim easily follows from 
 the existence of the natural short exact sequence of $\mathbb{F}[G_K]$-modules 
 $$0\rightarrow \mathrm{ad}(\overline{V})^0(\omega)\rightarrow
 \mathrm{ad}(\overline{V})(\omega)\xrightarrow{\mathrm{trace}}\mathbb{F}(\omega)
 \rightarrow 0,$$ which splits when $p\not|d$.
 
 We finally check that $\mathcal{X}_{\overline{V},\mathrm{reg}-\mathrm{cris}}$ is non empty under the 
 assumption that $\overline{V}$ is absolutely irreducible. This fact may be well known, but for convenience of readers, we recall a proof of this fact. 
 We use the notation used in the proof of Theorem \ref{1.17}.
 Let $K_d$ be the unramified extension of $K$ such that $[K_d:K]=d$.
 By extending $E$, we assume that $\mathrm{Hom}_{\mathbb{Q}_p-\mathrm{alg}}(K_d, E)=\mathrm{Hom}_{\mathbb{Q}_p-\mathrm{alg}}
 (K_d, \overline{\mathbb{Q}}_p)$. Fix $\widetilde{\sigma}_0:K_d\hookrightarrow E$ such that $\widetilde{\sigma}_0|_K=\sigma_0$. 
 Let $\chi_d:G_{K_d}^{\mathrm{ab}}\rightarrow \mathbb{F}^{\times}_{q^d}=
 (\mathcal{O}_{K_d}/\pi_K\mathcal{O}_{K_d})^{\times}$ be the fundamental character of 
 degree $d$, i.e. the character defined by $\chi_d(\mathrm{rec}_{K_d}(a)):=\bar{a}$ ($a\in\mathcal{O}_{K_d}^{\times}$), 
 $\chi_d(\mathrm{rec}_{K_d}(\pi_K))=1$. Then it is known that there exists an isomorphism
 $$\overline{V}\isom \mathrm{Ind}^{G_K}_{G_{K_d}}((\widetilde{\sigma}_0\circ \chi_d)^l)\otimes_{\mathbb{F}}\mathbb{F}(\eta)$$ for some 
 $l\in \mathbb{Z}$ and $\eta:G_K^{\mathrm{ab}}\rightarrow \mathbb{F}^{\times}$.
Because any $\eta$ has a crystalline 
 lift, we may assume that 
 $\overline{V}\isom \mathrm{Ind}^{G_K}_{G_{K_d}}((\widetilde{\sigma}_0\circ \chi_d)^i))$. 
Let 
 $\chi_{d,\mathrm{LT}}:G_{K_d}^{\mathrm{ab}}\rightarrow K_d^{\times}$ be the Lubin-Tate character of $K_d$ associated to 
 $\pi_K\in K\subseteq K_d$, and let $\tau$ be a generator of $\mathrm{Gal}(K_d/K)$. 
 For each $\sigma\in \mathcal{P}$, choose
 $\widetilde{\sigma}:K_d\rightarrow E$ a $\mathbb{Q}_p$-algebra homomorphism such that $\widetilde{\sigma}|_{K}=\sigma$. 
 Then, one can take some $\{a_{\sigma,i}\}_{\sigma\in \mathcal{P},0\leqq i\leqq d-1}\in \prod_{\sigma\in \mathcal{P},0\leqq i\leqq d-1}\mathbb{Z}$ 
 such that $a_{\sigma,i}\not=a_{\sigma,j}$ for any $i\not =j$ and that $$(\widetilde{\sigma}_0\circ \chi_d)^l\equiv \prod_{\sigma\in \mathcal{P},0\leqq i\leqq d-1}(\widetilde{\sigma}\tau^i\circ \chi_{d,\mathrm{LT}})^{a_{\sigma,i}} ( \text{mod } \pi_E ).$$ 
Then, $$\mathrm{Ind}^{G_K}_{G_{K_d}}(\prod_{\sigma\in \mathcal{P},0\leqq i\leqq d-1}(\widetilde{\sigma}\tau^i\circ \chi_{d,\mathrm{LT}})^{a_{\sigma,i}})$$ is a lift of $\overline{V}$ which is a crystalline representation whose $\sigma$-part of Hodge-Tate weights 
 is $\{a_{\sigma,i}\}_{0\leqq i\leqq d-1}$, i.e. this is an element of $\mathcal{X}_{\overline{V},\mathrm{reg}-\mathrm{cris}}$, which proves 
 the claim, hence proves the corollary.

 \end{proof}


\begin{thebibliography}{99}
\bibitem[BeCh09]{BeCh09}
 J. Bella\"iche,G. Chenevier, Families of Galois representations and Selmer groups,
Ast\'erisque 324, Soc. Math. France (2009).
\bibitem[Be08]{Be08}
L.Berger, Construction de ($\varphi,\Gamma$)-modules: repr\'esentations $p$-adiques et $B$-paires,  Algebra and Number Theory, 2 (2008), no. 1, 91--120.
\bibitem[Be-Co08]{Be-Co08}
L.Berger, P.Colmez, Familles de repr\'esentations de de Rham et monodromie p-adique, Ast\'erisque 319 (2008), 187-212.
\bibitem[Berk93]{Berk93}
V.Berkovich, \'Etale cohomology for nonarchimedian analytic spaces, Publications math\'ematiques de l'IHES 78 (1993).
\bibitem[BLR95]{BLR95}
 S.Bosch, W.L\"utkebohmert, M.Raynaud, Formal and rigid geometry III. The relative
maximum principle, Math. Ann. 302, 1-29 (1995).
\bibitem[Bu07]{Bu07}
 K. Buzzard, Eigenvarieties, in L-functions and Galois representations, 59-120, LMS Lecture
Note Ser., 320, Cambridge Univ. Press, Cambridge, 2007.
\bibitem[Ch09]{Ch09}
G.Chenevier, Une application des vari\'et\'es de Hecke des groupe unitaires, preprint.
\bibitem[Ch11]{Ch11}
G.Chenevier, On the infinite fern of Galois representations of unitary type. Annales Scientifiques de l'E.N.S. 44, 963-1019 (2011).


\bibitem[Ch13]{Ch13}
G.Chenevier, Sur la densit\`e des representations cristallines de $\mathrm{Gal}(\overline{\mathbb{Q}}_p/\mathbb{Q}_p)$, Math. Annalen 335, 1469-1525 (2013).
\bibitem[Co08]{Co08}
P.Colmez, Repr\'esentations triangulines de dimension 2, Ast\'erisque 319 (2008), 213-258.
\bibitem[Con99]{Con99}
B.Conrad, Irreducible components of rigid spaces. Ann. Inst. Fourier (Grenoble) 49 (1999), no. 2, 473-541.
\bibitem[Fo94]{Fo94}
J.-M. Fontaine, Le corps des p\'eriodes $p$-adiques,  Ast\'erisque 223 (1994), 59-111.
\bibitem[Ke08]{Ke08}
K.Kedlaya, Slope filtrations for relative Frobenius, Ast\'erisque 319 (2008), 259-301.
\bibitem[Ki03]{Ki03}
M.Kisin, Overconvergent modular forms and the Fontaine-Mazur 
conjecture, Invent. Math. 153 (2003), 373-454.
\bibitem[Ki08]{Ki08}
M.Kisin, Potentially semi-stable deformation rings, J.AMS, 21 (2) (2008), 513-546.

\bibitem[Ki10]{Ki10}
M.Kisin,  Deformations of $G_{\mathbb{Q}_p}$ and $\mathrm{GL}_2(\mathbb{Q}_p)$ representations, appendix of 
``P.Colmez, Repr\'esentations de $\mathrm{GL}_2(\mathbb{Q}_p)$ et $(\varphi,\Gamma)$-modules, Ast\'erisque 330 (2010), 281-509." 
\bibitem[Na09]{Na09}
K.Nakamura, Classification of two dimensional split trianguline representations of $p$-adic fields, Compositio Math. 145 (2009), 865-914.
\bibitem[Na10]{Na10}
K.Nakamura, Deformations of trianguline B-pairs and Zariski density of two dimensional crystalline representations, preprint arXiv:1006.4891 [math.NT].





\end{thebibliography}
\end{document}